\documentclass[11pt]{amsart}
\usepackage{fullpage}
\usepackage {amsmath}
\usepackage {amsthm}
\usepackage {enumerate}
\usepackage{amssymb}
\usepackage{amsfonts}
\usepackage{anysize}
\usepackage{setspace}
\usepackage[all]{xy}

\marginsize{3.15cm}{3.15cm}{1.6cm}{3cm}

\numberwithin{equation}{section}

\theoremstyle{plain}

\newtheorem{theorem}[equation]{Theorem}
\newtheorem{proposition}[equation]{Proposition}
\newtheorem{corollary}[equation]{Corollary}
\newtheorem{lemma}[equation]{Lemma}
\newtheorem{question}[equation]{Question}

\theoremstyle{definition}
\newtheorem{definition}[equation]{Definition}

\theoremstyle{definition}
\newtheorem{remark}[equation]{Remark}

\theoremstyle{definition}

\theoremstyle{definition}
\newtheorem{example}[equation]{Example}

\theoremstyle{definition}

\newenvironment{claim}[1]{\paragraph{{\it Claim #1:}}}{}
\newenvironment{proofclaim}{\paragraph{{\it Proof of Claim:}}}{\hfill$\blacksquare$\\}

\title{\scshape\bfseries Families of retractions and families of closed subsets on compact spaces.}

\author{{\bfseries S. Garc\'ia-Ferreira}}
\address{Centro de Ciencias Matem\'aticas\\
         Universidad Nacional Aut\'onoma de M\'exico\\
				 Campus Morelia\\
         Apartado Postal 61-3, Santa Mar\'ia, 58089, Morelia, Michoac\'an, M\'exico.}
\email{sgarcia@matmor.unam.mx, novo1126@hotmail.com}
\author{{\bfseries C. Yescas-Aparicio}}



\date{}

\thanks{Research of the first-named author was supported
PAPIIT grant no. IN-105318 }

\begin{document}

\begin{abstract} It is know that the Valdivia compact spaces can be characterized by a special family of retractions called $r$-skeleton (see \cite{kubis1}).  Also we know that there are compact spaces with $r$-skeletons which are not Valdivia.  In this paper, we shall study $r$-squeletons and special families of closed subsets of compact spaces. We prove that if $X$ is a zero-dimensional compact space and  $\{r_s:s\in \Gamma\}$ is an $r$-skeleton on $X$ such that $|r_s(X)| \leq \omega$ for all $s\in \Gamma$, then $X$ has a dense subset consisting of isolated points. Also we give conditions to an $r$-skeleton in order that this $r$-skeleton can be extended to an $r$-skeleton on  the  Alexandroff Duplicate of the base space. The standard definition of a Valdivia compact spaces is via a $\Sigma$-product of a power of the unit interval. Following this fact we introduce the notion of $\pi$-skeleton on a  compact space $X$ by embedding $X$ in a suitable power of the unit interval together with  a pair $(\mathcal{F},\varphi)$, where $\mathcal{F}$ is family of  metric separable subspaces of $X$ and $\varphi$ an $\omega$-monotone function which satisfy certain properties. This new notion generalize the idea of a $\Sigma$-product.
  We prove that a compact space admits a retractional-skeleton  iff it admits a $\pi$-skeleton. This equivalence allows to give a new proof of the fact that the product of compact spaces with retractional-skeletons admits an retractional-skeleton (see \cite{cuth1}). In \cite{casa1}, the Corson compact spaces are characterized by a special family of closed subsets. Following  this direction, we  introduce the notion of weak $c$-skeleton  which  under certain conditions  characterizes the Valdivia compact spaces and compact spaces with $r$-skeletons.
 \end{abstract}

\maketitle

\section{Introduction}

 The Valdivia and Corson compact spaces are spaces widely studied in the Abstract Functional Analysis. From the study of Valdivia compact spaces done in  \cite{kubis1} surged the concept of {\it $r$-skeleton} as a generalization of the Valdivia compact spaces,  this new concept is a family of retractions with certain properties:

\begin{definition}\label{skeleton}
Let  $X$ be a space. An $r$-{\it skeleton} on $X$ is a family $\{r_s:s\in \Gamma\}$ of retractions in $X$, indexed by an up-directed $\sigma$-complete partially ordered set $\Gamma$, such that
\begin{enumerate}
\item[$(i)$] $r_s(X)$ is a cosmic space, for each $s\in \Gamma$;

\item[$(ii)$] $r_s=r_s\circ r_t=r_t\circ r_s$ whenever $s\leq t$;

\item[$(iii)$] if $\langle s_n\rangle_{n<\omega}\subseteq \Gamma$ is an increasing sequence and $s=\sup \{s_n : n<\omega\}$, then $r_s(x)=\lim_{n\rightarrow \infty}r_{s_n}(x)$ for each $x\in X$; and

\item[$(iv)$] for each $x\in X$, $x=\lim_{s\in \Gamma}r_s(x)$.
\end{enumerate}
The {\it induced space} on $X$ defined by $\{r_s: s\in \Gamma\}$ is  $\bigcup_{s\in \Gamma}r_s(X)$. If $X=\bigcup_{s\in \Gamma}r_s(X)$, then we will say that $\{r_s:s\in \Gamma\}$ is a \textit{full $r$-skeleton}. Besides, if $r_s\circ r_t= r_t\circ r_s$, for any $s,t\in \Gamma$, then we will say that  $\{r_s:s\in \Gamma\}$ is \textit{commutative}.
\end{definition}

It is shown in  \cite{cuth1}  that the Corson compact spaces are those compact spaces which have a  full $r$-skeleton, and in the  paper \cite{kubis1} it is proved that the Valdivia compact spaces are  those compact spaces which   have a commutative $r$-skeleton. The ordinal space $[0,\omega_2]$ is a compact space with an $r$-skeleton, but it does not admit neither a commutative nor a full $r$-skeleton (see Example \ref{resqordinal}). Thus, the class of spaces with $r$-skeletons contains properly the Valdivia  compact spaces. In the literature, some other authors use the name  \textit{retractional skeleton} instead of $r$-skeleton, and the name \textit{non commutative Valdivia compact space} is used to address to a  compact space with an $r$-skeleton.

\medskip

In this article, we shall prove several results on compact spaces with an  $r$-skeletons as we describe in the next lines:

\medskip

$\bullet$ In the  third section,  we show that if $X$ is a  zero-dimensional compact space with an $r$-skeleton
$\{r_s:s\in \Gamma\}$  such that $|r_s(X)| \leq \omega$ for all $s\in \Gamma$, then $X$ has a dense subset consisting of isolated points.

$\bullet$ An interesting research topic has been the study of the connections between the topological properties of a space $X$ and its Alexandroff duplicate (denoted by $AD(X)$). Indeed, in the paper \cite{soma1} the author proved that if  $X$ be a compact space and $AD(X)$ has an $r$-skeleton, then $X$ has an $r$-skeleton. With respect to the reciprocal implication of this assertion, in the papers \cite{reynaldo1} and \cite{kalenda2}, it is shown that if  $X$ is a compact space with full $r$-skeleton, then $AD(X)$ has a full $r$-skeleton. Following this direction, one can wonder when an $r$-skeleton on a space can be extended to an $r$-skeleton on its Alexandroff duplicate. One particular case was is due to   J. Somaglia in \cite{soma1} who showed that if $X$ is  a compact space with  an $r$-skeleton and $Y$ the induced space such that $X\setminus Y$ is finite, then $AD(X)$ admits an $r$-skeleton.
 This result of  Somaglia motivated us to give  a necessary and sufficient conditions on the $r$-skeleton of a space $X$, in order that its Alexandroff duplicated $AD(X)$  admits an $r$-skeleton. All these results concerning the Alexandroff duplicate will be described in detail in the fourth  section.

$\bullet$ We recall that the Corson compact spaces are the compact spaces which are contained in a $\Sigma$-product, and the Valdivia compact spaces are the compact spaces whose intersection with a $\Sigma$-product is dense in the compact space.  Since these spaces can be also defined by $r$-skeletons, it is natural to investigate about a generalization of a $\Sigma$-product in order to obtain a characterization of the compact spaces which admit an $r$-skeleton.
This analysis is carried out in the fifth section, in fact we introduce the notion of {\it $\pi$-skeleton} to generalize the $\Sigma$-products and we prove that a compact space has a $\pi$-skeleton iff it has an $r$-skeleton.

$\bullet$ The notion of a full $c$-skeleton was introduced in the article  \cite{casa1} in which the authors proved that the Corson compact spaces are the spaces which admits a full $c$-skeleton.
Following this direction, in the sixth section,  we weaken the definition of $c$-skeleton to obtain the notion of a  weak $c$-skeleton. This concept plus some additional conditions provide  characterizations of both the Valdivia compact spaces and the compact spaces with an $r$-skeleton.

\section{Preliminaries}

Our spaces will be Tychonoff (completely regular and Hausdorff). The Greek letter $\omega$ will stand for the first infinite cardinal number. Given an infinite set  $X$,  the symbol $[X]^{\leq \omega}$ will denote the set of all countable  subsets of $X$ and the meaning of $[X]^{<\omega}$ should be clear. A partially ordered set $\Gamma$ is called {\it up-directed} whenever for every $s,s'\in \Gamma$, there is $t\in \Gamma$ such that $s\leq t$ and $s'\leq t$. And $\Gamma$ is named $\sigma$-{\it complete} if $\sup_{n<\omega}s_n\in \Gamma$, for each increasing sequence $ \langle s_n\rangle_{n<\omega} \subseteq \Gamma$.    For a space $X$,  the closure of $A \subseteq X$ will be denoted either by $\overline{A}$ or $cl_X(A)$. If $\alpha$ is an infinite ordinal number, $[0,\alpha]$ will denote the usual ordinal space. For a product $\Pi_{t\in T}X_t$, if $t\in  T$ and  $A\in [T]^{\leq \omega}$, then $\pi_t$ and $\pi_A$ will be denote the projection on the coordinate $t$ and the projection on the set   $A$, respectively. A subbasic open  of a product of spaces  $\Pi_{t\in T}X_t$ will be simply denoted by $[t,V] := \pi_t^{-1}(V)$, where  $t\in T$ and $V$ is an open subset of $X_t$. A space is called {\it cosmic} if it has a countable network. It is known that a compact space is cosmic iff it is metrizable and separable.

\medskip

The next property is useful to prove that a family of retractions is an $r$-skeleton without establishing explicitly the condition $(iv)$ of Definition \ref{skeleton}.

\begin{proposition}[\cite{casa1}]\label{rey2}
Let $\{r_s:s\in \Gamma\}$ be a family of retractions on a countably compact space $X$ satisfying  $(i)$-$(iii)$ of Definition \ref{skeleton}. If $Y=\bigcup_{s\in\Gamma}r_s(X)$, then $x=\lim_{s\in \Gamma}r_s(x)$,  for each $x\in \overline{Y}$.
\end{proposition}


The next example is a compact space with an $r$-skeleton which  is not Valdivia (other non-trivial examples of non-Valdivia compact spaces which admits $r$-skeletons are given in  \cite{soma2}).

\begin{example}[\cite{kubis1}]\label{resqordinal}
Let $\alpha>\omega_1$  and $\mathbf{C}_\alpha$ the family of all closed countable sets $A\subseteq [0,\alpha]$ such that $0\in A$ and if $p\in A$ is an isolated point in $A$, then $p$ is an isolated  point in $[0,\alpha]$. For each $A\in \mathbf{C}_\alpha$ we define $r_A:[0,\alpha]\rightarrow [0,\alpha]$ by $r_A(x)=\max\{a\in A: a\leq x\}$. The family $\{r_A:A\in \mathbf{C}_\alpha\}$ is a non commutative $r$-skeleton on $[0,\alpha]$.
\end{example}

\medskip

 The $r$-skeletons have been  very important  in the study of certain topological properties of compact spaces (see, for instance, \cite{cuth2}). Indeed,  in the next theorem,  we will highlight the  basic properties of the induced subspace that are  used in next sections.

\begin{theorem}[\cite{kubis}]
Let $X$ be a compact space with an $r$-skeleton and $Y\subseteq X$ be the  induced space by the $r$-skeleton. Then:
\begin{itemize}
\item $Y$ is  countably closed in $X$, i. e. for any $B\in [Y]^{\leq \omega}$ we have $\overline{B}\subseteq Y$,
\item  $Y$ is a Frechet-Urysohn space and
\item $\beta Y = X$.
\end{itemize}
\end{theorem}

An  $r$-skeleton of a space can be  inherited to  a  closed subspace:

\begin{theorem}[\cite{cuth1}]\label{teocuth1}
Let $X$ be a compact space with $r$-skeleton $\{r_s:s\in \Gamma\}$,  $Y\subseteq X$ be the  induced space by the $r$-skeleton and $F$ be a closed subset of $X$. If $Y\cap F$ is dense in $F$, then $F$ admits an $r$-skeleton.
\end{theorem}

In the proof of Theorem \ref{teocuth1}, the  author described a technique  to inherit an $r$-skeleton on a compact space to a closed subspace. The following remark enounce this technique.

\begin{remark}\label{obscuth}
Let $X$ be a compact space with $r$-skeleton $\{r_s:s\in \Gamma\}$, $Y$ the induced space by the $r$-skeleton and $F$ be a closed subset of $X$. If $Y\cap F$ is dense in $F$, then the set $\Gamma'=\{s\in \Gamma:r_s(F)\subseteq F\}$ is $\sigma$-complete and cofinal in $\Gamma$, and the family  $\{r_s\upharpoonright_{F}:s\in \Gamma'\}$ is an $r$-skeleton on $F$.
\end{remark}

\medskip

The next notion was a very  useful tool in the study of Corson compact spaces in the paper \cite{casa1}:

\smallskip

Given two  up-directed $\sigma$-complete partially ordered sets $\Gamma$ and $\Gamma'$, a function $\psi:\Gamma\rightarrow \Gamma'$ is called {\it $\omega$-monotone} provided that:
\begin{itemize}
\item $\psi(s)\leq \psi(t)$ whenever $s\leq t$; and
\item if $\langle s_n\rangle_{ n<\omega}\subseteq \Gamma$ is an increasing sequence, then $\psi(\sup_{n<\omega}s_n)=\sup_{n<\omega}\psi(s_n)$.
\end{itemize}

\medskip

  The $\omega$-monotone functions allows us to change the set of indices $\Gamma$ of a full $r$-skeleton to the set of countable subsets of the induced space (see \cite{reynaldo1}):

\begin{lemma}[\cite{reynaldo1}]\label{rey1}
Let $X$ be a set and $\Gamma$ be an up-directed partially ordered set. If for $x\in X$ there is an assignment $s_x\in \Gamma$, then  there exists an $\omega$-monotone function $\psi:[X]^{\leq \omega}\rightarrow \Gamma$ such that $s_x\leq \psi(x)$, for every $x\in X$.
\end{lemma}

We will use the replacement of $\Gamma$ by $[X]^{\leq \omega}$ in the proof of some results of this article.

\medskip

In \cite{casa1}, the authors use a special family  of closed subsets of a Corson compact spaces to define the notion of a  $c$-skeleton.
Thus, they obtained a characterization of a Corson compact spaces in terms of a $c$-skeleton. Let us state next this notion.

\begin{definition}[\cite{casa1}]\label{cesqueletodef}
Let $(X,\tau)$ be a space and $\{(F_s,\mathcal{B}_s):s\in \Gamma\}$ be a family of  pairs  such that $\{F_s:s\in \Gamma\}$ is a family of closed subsets of $X$ and $\{\mathcal{B}_s:s\in \Gamma\}\subseteq [\tau]^{\leq\omega}$. We say that $\{(F_s,\mathcal{B}_s):s\in \Gamma\}$ is a \textit{$c$-skeleton} if the following conditions hold:
\begin{itemize}
\item $F_s\subseteq F_t$, whenever $s\leq t$;
\item for each $s\in \Gamma$, $\mathcal{B}_s$ is a base for a topology  denoted by $\tau_s$ on $X$ and there exist a Tychonoff space $Z_s$ and continuous map $g_s:(X,\tau_s)\rightarrow Z_s$ which separates the points of $F_s$, and
\item the assignment $s\rightarrow \mathcal{B}_s$ is $\omega$-monotone.
\end{itemize}
In addition, if $X=\bigcup_{s\in \Gamma}F_s$, the we say that the $c$-skeleton is \textit{full}.
\end{definition}

 One of the main results of  \cite{casa1} asserts that  a compact space is Corson iff it admits  a full  $c$-skeleton. The next  compact non-Corson space  admits a $c$-skeleton.

\begin{example}\label{cesqejemplo} The ordered space
 $[0,\omega_1]$ admits a $c$-skeleton. Indeed, let $\Gamma=\omega_1$ and $L_{\omega_1}:=\{\alpha<\omega_1: \alpha \mbox{ is not isolated point}\}$. For each $\alpha<\omega_1$, define $$\mathcal{B}_\alpha:=\{\{\gamma\}:\gamma\leq \alpha \mbox{ and} \gamma\notin L_{\omega_1}\}\cup\{(\gamma,\beta+1):\gamma<\beta\leq \alpha \mbox{ and } \beta\in L_{\omega_1} \}\cup\{(\gamma,\omega_1]: \gamma<\alpha\}.$$   To see that the family   $\{([0,\alpha], \mathcal{B}_\alpha):\alpha<\omega_1\}$ is a $c$-skeleton on $[0,\omega_1]$ is enough to define the functions $g_\alpha$'s:

 For each $\alpha<\omega_1$, let $Z_\alpha=[0,\alpha]$ and consider the map $g_\alpha:([0,\omega_1],\mathcal{B}_\alpha)\rightarrow Z_\alpha$ defined  by $g_\alpha(\beta):=$
 \[
g_\alpha(\beta) :=
\begin{cases}
\beta & \textrm{ if $\beta<\alpha$},\\
\alpha & \textrm{ otherwise}, \\
\end{cases}
\]
for each $\beta\leq\omega_1$. Observe that the function $g_\alpha$ is continuous and one-to-one map, and  that the assignment $\alpha\rightarrow \mathcal{B}_\alpha$ is $\omega$-monotone.
\end{example}

\section{$r$-skeletons on zero-dimensional spaces}

It is very natural to ask when an $r$-skeleton $\{r_s:s\in \Gamma\}$ satisfies that $|r_s(X)| \leq \omega$ for each $s\in \Gamma$. In the next results, we give a necessary condition to have this property. In this section, we will use usual terminology of trees.    By $\{0,1\}$ we denote the discrete space with two elements. We will denote the Cantor space by $\mathcal{C}$, $\{0,1\}^{<\omega}$  is the set of all finite sequences of $\{0,1\}$; $\sigma\hat{\hspace{0.1cm}}\lambda$  will be denote the usual  concatenation   for $\sigma, \lambda\in \{0,1\}^{<\omega}$,  and an initial segment of $p\in \mathcal{C}$ of length $n$ will denoted by  $p\mid n$. For terminology not mentioned and used in the next results the reader may  consult \cite{kech}.

\begin{lemma}\label{lema2}

Let $X$ be a zero-dimensional compact space without isolated points. If $\{r_s:s\in \Gamma\}$ is an $r$-skeleton on $X$, then there is a Cantor scheme $\{U_\sigma:\sigma\in \{0,1\}^{<\omega}\}$ and $\{x_\sigma:\sigma\in \{0,1\}^{<\omega}\}\subset\bigcup_{s\in\Gamma}r_s(X)$ such that:
\begin{enumerate}
\item $U_\sigma$ is a clopen set,
\item $x_\sigma\in U_\sigma$ and for $i\in \{0,1\}$, $x_{\sigma}\notin U_{\sigma\hat{\hspace{0.1cm}}i}$.
\end{enumerate}
Moreover, $\bigcap_{n<\omega}U_{p\mid n}\neq\emptyset$, for every $p\in \mathcal{C}$.
\end{lemma}

\begin{proof}
Let $Y$ the induced space by $\{r_s:s\in\Gamma\}$. We proceed by induction on the length of $\sigma$. Define $U_\emptyset:=X$ and choose a point $x_\emptyset\in Y$. Now, we suppose that  $U_\sigma$ and $x_\sigma$ are defined for all $\sigma\in \{0,1\}^{<\omega}$ of length $n$. Take $\sigma\in \{0,1\}^{<\omega}$ and suppose that length of $\sigma$ is $n$. Since $x_\sigma$ is not isolated, then there are $z, z'\in U_\sigma$ not equals to $x_\sigma$. Pick two disjoint clopen subsets  $U$ and $V$ which $z\in U$, $z'\in V$, $x_\sigma\notin U$ and $x_\sigma\notin V$. Let be $U_{\sigma\hat{\hspace{0.1cm}}0}:= U_\sigma\cap U$ and   $U_{\sigma\hat{\hspace{0.1cm}}1}:= U_\sigma\cap V$. Using the density of $Y$, we can choose $x_{\sigma\hat{\hspace{0.1cm}}0}\in U_{\sigma\hat{\hspace{0.1cm}}0}\cap Y$ and $x_{\sigma\hat{\hspace{0.1cm}}1}\in U_{\sigma\hat{\hspace{0.1cm}}1}\cap Y$. Hence, $\{U_\sigma:\sigma\in \{0,1\}^{<\omega}\}$ and $\{x_\sigma:\sigma\in \{0,1\}^{<\omega}\}$ are as we expected.   Now, pick $p\in \mathcal{C}$.   By construction,  $\{U_{p\mid n}:n< \omega\}$ has the finite intersection property. By compactness, it follows that $\bigcap_{n<\omega}U_{p\mid n}\neq\emptyset$.
\end{proof}

\begin{theorem}\label{isol}
Let $X$ be a zero-dimensional compact space without isolated points. If $\{r_s:s\in \Gamma\}$ is an $r$-skeleton on $X$, then there is $s\in \Gamma$  such that $r_s(X)$ is not countable.
\end{theorem}

\begin{proof}
Let $Y$ be the induced space by $\{r_s:s\in\Gamma\}$ and we consider $\{U_\sigma:\sigma\in \{0,1\}^{<\omega}\}$ and $\{x_\sigma:\sigma \in \{0,1\}^{<\omega}\}\subset\bigcup_{s\in\Gamma}r_s(X)$ as in Lemma \ref{lema2}.
Put  $F=\{x_\sigma:\sigma\in \{0,1\}^{<\omega}\}$. Since $F\in [Y]^{\leq \omega}$ and $Y$ is countably closed, then $\overline{F}\subseteq Y$. We shall prove that $\overline{F}$ is not countable. For $p\in \mathcal{C}$, let  $V_p=\bigcap_{n<\omega}U_{p\mid n}$ and $x_p\in \overline{\{x_{p\mid n}:n<\omega\}}$.
We observe that $x_p\notin\{x_{p\mid n}:n<\omega\}$. Since $\{x_{p\mid n}:n<\omega\}\in [Y]^{\leq \omega}$ and $Y$ is countably closed, we have $x_p\in Y$ and it is easy to see that $x_p\in V_p$. Finally, if $p,q\in \mathcal{C}$ with $p\neq q$, using that $V_p\cap V_q=\emptyset$ we have $x_p\neq x_q$. Since $\{x_p:p\in \mathcal{C}\}\subseteq \overline{F}$, we conclude that $\overline{F}$ is not countable. Finally, we choose $s\in \Gamma$ such that $F\subseteq r_s(X)$.
\end{proof}

\begin{corollary}\label{corisol}
 Let $X$ be a  zero-dimensional compact space. If $\{r_s:s\in \Gamma\}$ is an $r$-skeleton on $X$ such that $|r_s(X)| \leq \omega$ for all $s\in \Gamma$, then $X$ has a dense subset consisting of isolated points.
\end{corollary}

\begin{proof}
Let $Y$ be the induced space by $\{r_s:s\in\Gamma\}$ and we suppose that $U\subseteq X$ is an open subset without isolated points. Take a nonempty open subset $W$ of $X$ such that $\overline{W}\subseteq U$. We have that $Y\cap \overline{W}$ is dense in $\overline{W}$. By Theorem \ref{teocuth1}, $\overline{W}$ admits an $r$-skeleton. Moreover, from Remark \ref{obscuth}, the $r$-skeleton is a subfamily of restrictions on $\overline{W}$ of the family $\{r_s:s\in \Gamma\}$. Using Proposition \ref{isol}, we obtain a contradiction.
\end{proof}

\begin{example}
Let $\alpha$ be a cardinal number with $\alpha \geq \omega_1$. The space $[0,\alpha]$ with the $r$-skeleton given in \ref{resqordinal}, is an example that satisfies  the conditions of  Corollary \ref{corisol}. Besides, Somaglia  proved in \cite{soma1} that  the  Alexandroff duplicate of $[0,\omega_2]$ also admits such $r$-skeleton.
\end{example}

\begin{example}
In the paper \cite{reynaldo1}, the authors  proved that the Alexandroff duplicate of a Corson space its again a Corson space, they extended the full $r$-skeleton of a Corson space to a full $r$-skeleton on its Alexandroff Duplicate. From the proof of this result  we deduce that  there are Corson spaces  $X$ whose  Alexandroff duplicate  contains a dense set of isolated points and admits an $r$-skeleton, but it could  fail  that $|r_s(AD(X))| \leq \omega$ for every $s\in \Gamma$, this shows that condition of Corollary \ref{corisol} is not sufficient.
\end{example}

\begin{question}
Which compact spaces with a dense set of isolated points  admit an $r$-skeleton whose retraction have countable images?
\end{question}


\section{$r$-skeletons and the Alexandroff Duplicate}


In this section, we study some topological  properties of the Alexandroff duplicate of a compact space with an $r$-skeleton. In particular, we give conditions to extend an $r$-skeleton on $X$ to an $r$-skeleton on  the Alexandroff duplicate $AD(X)$. To have this done we state some preliminary results.

\medskip
Remember that the \textit{Alexandroff duplicate} of a space $X$, denoted by $AD(X)$,  is  the space $X \times \{0,1\}$ with the topology in which all points of $X \times \{1\}$ are isolated, and the basic neighborhoods of the points $(x,0)$ are of the form $(U \times \{0,1\}) \setminus \{(x,1)\}$ where $U$ is a neighborhood of $x\in X$. We denote by $X_i$ the subspace $X\times \{i\}$. We denote by $\pi$ the projection from $AD(X)$ onto $X$. Remember that $X$ is homeomorphic to $X_0$ which, in some cases, will be identified with $X$.

\medskip

The next result appears implicitly in the development of several articles on the subject, here we provide a proof of it.
\begin{theorem}\label{numerables}
Let $X$ be a compact space, $\{r_s:s\in \Gamma\}$ be an $r$-skeleton on $X$ and $Y$ be the induced space. Then there is an $r$-skeleton $\{R_A: A\in [Y]^{\leq \omega}\}$ on $X$ such that
\begin{itemize}
\item $A\subseteq R_A(X)$, for all $A\in [Y]^{\leq \omega}$, and
\item $Y=\bigcup_{A\in [Y]^{\leq \omega}} R_A(X)$.
\end{itemize}
\end{theorem}
\begin{proof}
For each $x\in Y$, we pick $s_x\in \Gamma$ such that $r_{s_x}(x)=x$.  We consider the $\omega$-monotone function $\psi:[Y]^{\leq \omega}\rightarrow \Gamma$  given by Lemma \ref{rey1}.
Now, for each  $A\in [Y]^{\leq \omega}$ we define $R_A:X\rightarrow X$ by $R_A:=r_{\psi(A)}$.
 We shall prove that the family $\{R_A:A\in [Y]^{\leq \omega}\}$ is an $r$-skeleton on $X$.  Since   $\psi$  is $\omega$-monotone,  $s_x\leq\psi(x)$, for every $x\in Y$, and $\{r_s:s\in \Gamma\}$ is an $r$-skeleton. Trivially, we have that  $\{R_A:A\in [Y]^{\leq \omega}\}$ satisfies $(i)-(iii)$ of the $r$-skeleton definition.
 Finally, by Lemma \ref{rey2},   we have that $x=\lim_{A\in [Y]^{\leq \omega}}R_A(x)$ for all $x\in \overline{\bigcup_{A\in [Y]^{\leq \omega} }R_A(X)}$. By the choice of the family $\{s_x:x\in Y\}$, we  obtain that $A\subseteq R_A(X)$, for all $A\in [Y]^{\leq \omega}$. It then follows that $\bigcup_{A\in [Y]^{\leq \omega}}R_A(X)=Y$. Therefore, $\{R_A: A\in [Y]^{\leq \omega}\}$ is an $r$-skeleton on $X$.
\end{proof}
The next Lemma is a necessary condition  when  the Alexandroff duplicate admits an $r$-skeleton.
\begin{lemma}\label{teo1}
Let $X$ be a compact space, suppose that $AD(X)$ has an $r$-skeleton  with induced space $\hat{Y}$ and  $Y=\pi(\hat{Y}\cap X_0)$. Then for any $B\in [X\setminus Y]^{\leq \omega}$ we have that
\begin{itemize}
\item[$(*)$]   $cl_X(B)\setminus B\subseteq Y$ and
\item[$(**)$]   $cl_X(B)\setminus B$  is a cosmic space.
\end{itemize}
\end{lemma}
\begin{proof}
By  Theorem \ref{teo1}, we may consider an $r$-skeleton on $AD(X)$ of the form $\{r_C:C\in [\hat{Y}]^{\leq \omega}\}$. Let $B\in [X\setminus Y]^{\leq \omega}$. If $B$ is a finite set, then  the result follows immediately. Let us suppose that $|B|=\omega$ and $p\in cl_X(B)\setminus B$. Observe that $(p,0)\in cl_{AD(X)}(B\times\{1\})$. Since $B\times\{1\}\subseteq\hat{Y}$ and $\hat{Y}$ is countably closed, it follows that $(p,0)\in \hat{Y}$ and so $p\in Y$. Hence, $cl_X(B)\setminus B\subseteq Y$ and  then we deduce that $B$ is discrete in $X\setminus Y$.  Now, we note that $(cl_X(B)\setminus B)\times \{0\}\subseteq cl_{AD(X)}(B\times\{1\})$. Since  $B\times \{1\}\in [\hat{Y}]^{\leq \omega}$, we have that $cl_{AD(X)}(B\times\{1\})\subseteq r_{B\times \{1\}}(X)$. Since $r_{B\times \{1\}}(X)$ is cosmic, we conclude that $cl_X(B)\setminus B\subseteq Y$ is also cosmic.
\end{proof}

The next lemma is the key to extend a retraction from the base space to its Alexandroff duplicated.

\begin{lemma}\label{lemma51}
Let $X$ be a compact space which admits an $r$-skeleton  $\{r_s:s\in \Gamma\}$ with induced space $Y$. Suppose that   for   $B\in [X\setminus Y]^{\leq \omega}$ the conditions $(*)$ and $(**)$ from above hold. Then for each $s\in \Gamma$ such that $cl_X(B)\setminus B\subseteq r_s(X)$ and for every $A\in [r_s(X)]^{\leq\omega}$, the mapping  $R_{(A,B,s)}:AD(X)\rightarrow AD(X)$ defined as
\[
R_{(A,B,s)}(x,i) :=
\begin{cases}
(x,1) & \textrm{ if $x \in A\cup B$ and $i=1$},\\
(r_{s}(x),0) & \textrm{ otherwise}, \\
\end{cases}
\]
 for every $(x,i)\in AD(X)$, is a retraction on $AD(X)$.
\end{lemma}

\begin{proof}
Let $A\in [r_s(X)]^{\leq\omega}$ and assume that $cl_X(B)\setminus B\subseteq r_s(X)$. First, we prove that   $R_{(A,B,s)}$ is  a continuous function. Let $\langle(x_\lambda,i_\lambda) \rangle_{\lambda\in \Lambda} $ be a net such that $(x_\lambda,i_\lambda)\rightarrow (x,i)$. If $i=1$, then $\langle (x_\lambda,i_\lambda) \rangle_{\lambda\in \Lambda} $ is eventually constant and hence  $\langle R_{(A,B)}(x_\lambda,i_\lambda)\rangle_{\lambda\in \Lambda} $  is so.  Now, we consider the case when $i=0$ and assume that $\langle(x_\lambda,i_\lambda) \rangle_{\lambda\in \Lambda} $ is not trivial. Hence,  $R_{(A,B,s)}(x,0)=(r_s(x),0)$. We may suppose that either $\langle (x_\lambda,i_\lambda) \rangle_{\lambda\in \Lambda}\subseteq X_0$ or $\langle (x_\lambda,i_\lambda) \rangle_{\lambda\in \Lambda}\subseteq X_1$. First, we consider the case when $\langle (x_\lambda,i_\lambda) \rangle_{\lambda\in \Lambda}\subseteq X_0$, we have that  $R_{(A,B,s)}(x_\lambda,i_\lambda)=(r_s(x_\lambda),0)$ for all $\lambda \in \Lambda$, and since $r_s$ continuous, we obtain that $ R_{(A,B,s)}(x_\lambda,i_\lambda)\rightarrow R_{(A,B,s)}(x,0)$. Now,  assume that $\langle (x_\lambda,i_\lambda) \rangle_{\lambda\in \Lambda}\subseteq X_1$. If $\langle x_\lambda \rangle_{\lambda\in \Lambda}$ contains a subnet that lies eventually  in $A\cup B$, then $x\in cl_X(A\cup B)$.  Without loss of generality, we suppose that $x_\lambda \in A \cup B$ for every $\lambda \in \Lambda$. Then, $R_{(A,B,s)}(x_\lambda,i_\lambda)=(x_\lambda,1)$ for all $\lambda\in \Lambda$.  If $x\in B$,  then we deduce from  $(*)$ that $\langle x_\lambda\rangle_{\lambda \in \Lambda}$ is eventually in $A$. Since $A\subseteq r_s(X)$, we have that $x\in r_s(X)$.  As $Y\cap B=\emptyset$, $x\in cl_X(A)\cup \big( cl_X(B)\setminus B\big)$. From $A\cup \big( cl_X(B)\setminus B\big) \subseteq r_s(X)$ we deduce that $r_s(x)=x$ and we conclude  that $R_{(A,B,s)}(x_\lambda,i_\lambda)\rightarrow R_{(A,B,s)}(x,0)$. In other hand, $\langle x_\lambda \rangle_{\lambda\in \Lambda}\subseteq X\setminus \big(A\cup B\big)$ and $R_{(A,B,s)}(x_\lambda,i_\lambda) = (r_{s}(x_\lambda),0)$. Using the continuity of $r_{s}$, we have that $R_{(A,B,s)} (x_\lambda,i_\lambda)=(r_s(x_\lambda),0)\rightarrow  (r_s(x),0)=R_{(A,B,s)} (x,0)$. Therefore, $R_{(A,B,s)}$ is  a continuous function. Finally, if $(x,i)\in AD(X)$, then
\begin{align*}
R_{(A,B,s)}\circ R_{(A,B,s)}(x,i) &=
\begin{cases}
R_{(A,B,s)}(x,1) & \textrm{ if $x \in A\cup B$ and $i=1$},\\
R_{(A,B,s)}(r_{s}(x),0) & \textrm{ otherwise.} \\
\end{cases}\\
&=
\begin{cases}
(x,1) & \textrm{ if $x \in A\cup B$ and $i=1$},\\
(r_{s}(r_{s}(x)),0) & \textrm{ otherweise.} \\
\end{cases}\\
&=
\begin{cases}
(x,1) & \textrm{ if $x \in A\cup B$ and $i=1$},\\
(r_{s}(x),0) & \textrm{ otherwise.} \\
\end{cases}\\
&=R_{(A,B,s)}(x,i).
\end{align*}
That is, $R_{(A,B,s)}$ is a retract on $AD(X)$.
\end{proof}

\medskip

The main theorem of  this article is the following.

\medskip

 In  what follows, we will consider    $\sigma$-complete up-directed partially ordered  sets $\Gamma$, where $\Gamma$ will be a subset of  $[Y]^{\leq \omega}\times[X\setminus Y]^{\leq \omega}$ with the order $\preceq$ defined by $(A,B)\preceq (A',B')$ if $A\subseteq A'$ and $B\subseteq B'$. Also, we remark that in this partially order set  $\Gamma$  if $\langle(A_n,B_n)\rangle_{ n<\omega}\subseteq \Gamma$, then $\sup_\Gamma\langle(A_n,B_n)\rangle_{ n<\omega}$ is not necessarily $(\bigcup_{n<\omega}A_n,\bigcup_{n<\omega}B_n)$.

\begin{theorem}\label{principal}
Let   $X$ be a compact space.
$AD(X)$ admits an $r$-skeleton if and only if  there is an $r$-skeleton $\{r_{(A,B)}:(A,B)\in \Gamma\}$ on $X$ with induced space $Y$ such that $\Gamma\subseteq [Y]^{\leq \omega}\times [X\setminus Y]^{\leq \omega}$  is $\sigma$-complete and cofinal in $[Y]^{\leq \omega}\times [X\setminus Y]^{\leq \omega}$, and  the next conditions hold:

\noindent  For every $B\in [X\setminus Y]^{\leq \omega}$,
\begin{itemize}
\item[$(*)$]  $cl_X(B)\setminus B\subseteq Y$,
\item[$(**)$]  $cl_X(B)\setminus B$ is cosmic;
\end{itemize}
 and
\begin{itemize}
\item[$(***)$]  $A\subseteq r_{(A,B)}(X)$ and  $cl_X(B)\setminus B\subseteq r_{(A,B)}(X)$ for every $(A,B)\in \Gamma$.
\end{itemize}
\end{theorem}
\begin{proof}
Necessity. Suppose that $AD(X)$ admits an $r$-skeleton with induced space $\hat{Y}$. Observe that $X_1\subseteq \hat{Y}$. According to Theorem \ref{numerables}, we may assume that such $r$-skeleton  is of the form  $\{R_C:C\in [\hat{Y}]^{\leq \omega}\}$ and satisfies the properties of the theorem. Consider the set $Y=\pi(\hat{Y}\cap X_0)$  and put $\Gamma'=[Y]^{\leq \omega}\times [X\setminus Y]^{\leq \omega}$.  For $(A,B)\in \Gamma'$, we define $R_{(A,B)}=R_{A\times \{0,1\}\cup B\times\{1\}}$. We claim that $\{R_{(A,B)}:(A,B)\in \Gamma'\}$ is an $r$-skeleton on $AD(X)$ with induced space $\hat{Y}$. The conditions $(i), (ii)$ and $(iii)$  hold because $\{R_C:C\in [\hat{Y}]^{\leq \omega}\}$ is an $r$-skeleton. For the condition $(iv)$, let $(x,i)\in \hat{Y}$ and choose $(A,B)\in \Gamma'$ so that $(x,i)\in A\times \{0,1\}\cup B\times\{1\}$. It follows that $(x,i)=R_{A\times \{0,1\}\cup B\times\{1\}}(x,i)=R_{(A,B)}(x,i)$.
Thus, we have proved that  $\{R_{(A,B)}:(A,B)\in \Gamma'\}$ is an $r$-skeleton on $AD(X)$ with induced space $\hat{Y}$. By Remark \ref{obscuth},
 the set $\Gamma=\{(A,B)\in \Gamma':R_{(A,B)}(X_0)\subseteq X_0\}$ is $\sigma$-complete and cofinal in $\Gamma'$ and $\{R_{(A,B)}\upharpoonright_{X_0}:(A,B)\in \Gamma\}$ is an $r$-skeleton on $X_0$, with induced  space  $\hat{Y}\cap X_0$. For each $(A,B)\in \Gamma$, we define $r_{(A,B)}=\pi(R_{(A,B)}\upharpoonright_{X_0})$. Hence,  $\{r_{(A,B)}:(A,B)\in \Gamma\}$ is an $r$-skeleton on $X$ with induced space $Y$. By Lemma \ref{teo1}, the conditions $(*)$ and $(**)$  hold.  The  condition $(***)$ is easy to verify.\\

Sufficiency. Now, let $\{r_{(A,B)}:(A,B)\in \Gamma\}$ an $r$-skeleton on $X$ which satisfies the condition $(*)-(***)$, where $\Gamma\subseteq [Y]^{\leq \omega}\times [X\setminus Y]^{\leq \omega}$  is $\sigma$-complete and cofinal in $[Y]^{\leq \omega}\times [X\setminus Y]^{\leq \omega}$ and $Y$ is the induced space.
For each $(A,B)\in \Gamma$, we know that $cl_X(B)\setminus B\subseteq r_{(A,B)}(X)$ and $A\subseteq r_{(A,B)}(X)$. Set $R_{(A,B)}=R_{(A,B,(A,B))}$, where  $R_{(A,B,(A,B))}$  is the retraction of  Lemma \ref{lemma51}.
 We claim that $\{R_{(A,B)}:(A,B)\in \Gamma\}$ is an $r$-skeleton on $AD(X)$. Indeed, we shall prove that the conditions $(i)-(iv)$ of the $r$-skeleton definition  hold.
\begin{enumerate}
\item[$(i)$] If  $(A,B)\in \Gamma$, then $R_{(A,B)}(AD(X))=(r_{(A,B)}(X)\times \{0\})\cup ((A\cup B) \times \{1\})$ is a cosmic space.
\item[$(ii)$] Let  $(A,B)\preceq (A',B')$. Fix $(x,i)\in AD(X)$. Then
\begin{align*}
R_{(A,B)}\circ R_{(A',B')}(x,i)&=
\begin{cases}
R_{(A,B)}(x,1) & \mbox{ if }x \in A'\cup B'\mbox{ and }i=1,\\
R_{(A,B)}(r_{(A',B')}(x),0) & \mbox{ otherwise.} \\
\end{cases} \\
&
=
\begin{cases}
(x,1) & \mbox{ if }x \in A\cup B\mbox{ and }i=1,\\
(r_{(A,B)}(x),0) & \mbox{ if }x \in (A'\cup B')\setminus (A\cup B) \mbox{ and }i=1,\\
(r_{(A,B)}(r_{(A',B')}(x)),0) & \mbox{ otherwise.} \\
\end{cases} \\
&=
\begin{cases}
(x,1) & \mbox{ if }x \in A\cup B\mbox{ and }i=1,\\
(r_{(A,B)}(x),0) & \mbox{ otherwise.} \\
\end{cases} \\
&= R_{(A,B)}(x,i).
\end{align*}

And we also have that
\begin{align*}
R_{(A',B')}\circ R_{(A,B)}(x,i)&=
\begin{cases}
R_{(A',B')}(x,1) & \mbox{ if }x \in A\cup B\mbox{ and }i=1,\\
R_{(A',B')}(r_{(A,B)}(x),0) & \mbox{ otherwise.} \\
\end{cases} \\
&=
\begin{cases}
(x,1) & \mbox{ if }x \in A\cup B\mbox{ and }i=1,\\
(r_{(A',B')}(r_{(A,B)}(x)),0) & \mbox{ otherwise.} \\
\end{cases} \\
&=
\begin{cases}
(x,1) & \mbox{ if }x \in A\cup B\mbox{ and }i=1,\\
(r_{(A,B)}(x),0) & \mbox{ otherwise.} \\
\end{cases} \\
&= R_{(A,B)}(x,i).
\end{align*}
 Therefore, $R_{(A,B)}=R_{(A,B)}\circ R_{(A',B')}=R_{(A',B')}\circ R_{(A,B)}$ whenever  $(A,B)\preceq (A',B')$.

\item[$(iii)$] Let $\langle(A_n,B_n)\rangle_{ n<\omega}\subseteq \Gamma$ be such that $(A_n,B_n)\preceq(A_{n+1},B_{n+1})$.  $cl_X(B)\setminus B\subseteq Y$. For simplicity, put  $(A,B)=\sup\{ (A_n,B_n): n<\omega \}$. Fix $(x,i)\in AD(X)$. We will prove that $R_{(A,B)}(x,i)=\lim_{n\rightarrow \infty}R_{(A_n,B_n)}(x,i)$. In fact, if $i=0$, then $R_{(A,B)}(x,0)=(r_{(A,B)}(x),0)$ and $R_{(A_n,B_n)}(x,0)=(r_{(A_n,B_n)}(x),0)$ for all $n < \omega$. Since   $r_{(A,B)}(x)=\lim_{n\rightarrow\infty}r_{(A_n,B_n)}(x)$, we conclude that $R_{(A,B)}(x,0)=  \lim_{n\rightarrow \infty}R_{(A_n,B_n)}(x,0)$.
Now, we consider the case when $i=1$. If $x\in A\cup B$, then there is $n_0<\omega$ such that  $x\in A_n\cup B_n$ for all $n\geq n_0$.
 Hence, $R_{(A,B)}(x,1)=(x,1)=R_{(A_n,B_n)}(x,1)$, for  every $n\geq n_0$.  It  then follows that $R_{(A,B)}(x,1)=  \lim_{n\rightarrow \infty}R_{(A_n,B_n)}(x,1)$.
  If $x\notin A\cup B$, then $R_{(A,B)}(x,1)=  (r_{(A,B)}(x),0)$ and $R_{(A_n,B_n)}(x,1)=(r_{(A_n,B_n)}(x),0)$, for every $n < \omega$.
  Since the equality $r_{(A,B)}(x)=\lim_{n\rightarrow\infty}r_{(A_n,B_n)}(x)$ holds,
$R_{(A,B)}(x,1)=  \lim_{n\rightarrow \infty}R_{(A_n,B_n)}(x,1).$

\item[$(iv)$] Let $(x,i)\in AD(X)$. First, if $i=0$, then we notice that the equality $x=\lim_{(A,B)\in \Gamma}r_{(A,B)}(x)$ implies that $(x,0)=\lim_{(A,B)\in \Gamma}R_{(A,B)}(x,0)$. Now, we su\-ppose $i=1$ and let  $x\in Y$. By cofinality of $\Gamma$, there is $(A,B)\in \Gamma$ such that $\{x\}\subseteq A$. Since $x\in A\cup B$, we have $(x,1)=R_{(A,B)}(x,1)$. Now, let $x\in X\setminus Y$. By using  the confinality of $\Gamma$, there is $(A,B)\in \Gamma$ such that $\{x\}\subseteq B$.  It follows that $(x,1)=R_{(A,B)}(x,1)$.  Therefore, $(x,i)=\lim_{(A,B)\in \Gamma}R_{(A,B)}(x,i)$, for each $i\in \{0,1\}$.
\end{enumerate}

\end{proof}
 From the proof of the previous  theorem, we  deduce the next corollary.
\begin{corollary}\label{principalvaldivia}
   Let   $X$ be a compact space such that $AD(X)$ admits a commutative $r$-skeleton.  Then   the $r$-skeleton   $\{r_{(A,B)}:(A,B)\in \Gamma\}$ on $X$ obtained in the Theorem \ref{principal} is commutative.
\end{corollary}

It is well known (see \cite{kalenda2}) that there are  Valdivia compact spaces whose their  Alexandroff duplicates are not Valdivia compact.
If the $r$-skeleton   $\{r_{(A,B)}:(A,B)\in \Gamma\}$ on $X$ given in  Theorem  \ref{principal} is commutative, we do not know whether it can be extended to  a commutative $r$-skeleton on $AD(X)$.

If we have a commutative $r$-skeleton  of the form $\{r_{(A,B)}:(A,B)\in \Gamma\}$ on $X$,  the conditions given in the Theorem  \ref{principal} are not  clear for extend  to a commutative $r$-skeleton on $AD(X)$. In the next result, we add one more condition in order that we can  extend commutative $r$-skeletons.

\begin{corollary}
Let   $X$ be a compact space and  $\{r_{(A,B)}:(A,B)\in \Gamma\}$ a commutative $r$-skeleton  on $X$ as in the Theorem \ref{principal} which satisfies
\begin{itemize}
\item[$(****)$] for every $(A,B),(A',B')\in \Gamma$, $r_{(A,B)}(x)=r_{(A',B')}(r_{(A,B)}(x))$, for each $x\in B'\setminus B$.
\end{itemize}
Then $AD(X)$ admits a commutative $r$-skeleton.
\end{corollary}

\begin{corollary}
Let $X$ be a compact space. If $AD(X)$ admits an $r$-skeleton, then  the induced space $Y=\pi(\hat{Y}\cap X_0)$ of $X$ is unique. That is, if $Y'$ is an induced space by an arbitrary $r$-skeleton on $X$, then $Y'=Y$.
\end{corollary}

\begin{proof}
Let $Y'$ be a subset of $X$ induced by an $r$-skeleton on $X$. We suppose that $Y\neq Y'$.  By Lemma 3.2 from \cite{cuth1}, we have that  $Y\cap Y'$ cannot be dense in $X$. Hence, there is a nonempty  open subset $V$ of $X$ such that $V\cap (Y\cap Y')=\emptyset$. Let $W$ a nonempty open subset such that $cl_X(W)\subseteq V$. By density of $Y'$, there is an infinite  countable set $B\subseteq W \cap Y'$ . Since  $Y'$  is countably closed,  we must have $cl_{X}(B)\subseteq cl_X(W)\cap Y'\subseteq V\cap Y'$. On the other hand,  Theorem \ref{principal} implies that $cl_{X}(B)\setminus B\subseteq Y$, but this is impossible since  $V\cap (Y\cap Y')=\emptyset$. Therefore, $Y=Y'$.
\end{proof}

As a consequence of the previous corollary,  if $AD(X)$ is not Corson compact space and has an $r$-skeleton, then $X$ is not  a super Valdivia\footnote{We say that a compact space $X$ is {\it super Valdivia} if for every $x\in X$ there is a  dense $\Sigma$-subset $Y$ of $X$ such that $x\in Y$ (see \cite{kalenda1}).} space. In particular, Alexandroff duplicate  of $[0,1]^{\kappa}$ does not admit  an $r$-skeleton, for every $\kappa\geq \omega_1$.

\bigskip
In terms of $\omega$-monotone functions, we have the next result.

\begin{proposition}\label{novo2}
Let $X$ be a compact space which admits an $r$-skeleton with induced space $Y$. Let us suppose that  $\{r_A:A\in [Y]^{\leq \omega}\}$ is the $r$-skeleton obtained by Theorem \ref{numerables}. If the conditions $(*)$ and $(**)$ from above hold,  and  there is $\psi:[X\setminus Y]^{\leq\omega}\rightarrow [Y]^{\leq \omega}$ $\omega$-monotone such that for $B\in [X\setminus Y]^{\leq \omega}$, $cl_X(B)\setminus B \subseteq r_{\psi(B)}(X)$. Then $AD(X)$ has an $r$-skeleton.
\end{proposition}

\begin{proof}
Let $\Gamma'=[Y]^{\leq \omega}\times[X\setminus Y]^{\leq \omega}$. For each $(A,B)\in \Gamma'$, let $r_{(A,B)}:X\rightarrow X$ the function defined  by $r_{(A,B)}=r_{A\cup \psi(B)}$. The family $\{r_{(A,B)}:(A,B)\in \Gamma'\}$ is an $r$-skeleton on $X$ which satisfies the conditions $(*)-(***)$. Therefore, using Theorem \ref{principal}, $AD(X)$ has an $r$-skeleton.
\end{proof}

 We do not know whether or not the monotony of  Proposition \ref{novo2} is sufficient.

\medskip

The next is an example of application of Theorem \ref{principal}.
\begin{example}
Let $\kappa$ an uncountable  cardinal number. We considerer the $r$-skeleton $\{r'_A:A\in \mathbf{C}_\kappa\}$ given  in the Example \ref{resqordinal}. For this $r$-skeleton, $Y=\bigcup\mathbf{C}_\kappa$ is the induced space. By properties of the ordinal space $[0,\kappa]$ it follows that $(*)$ and $(**)$  hold. By Theorem \ref{numerables}, from $\{r'_A:A\in \mathbf{C}_\kappa\}$ we obtain an $r$-skeleton $\{r_A:A\in [Y]^{\leq \omega}\}$ with  induced space $Y$ that satisfies $(*)$ and $(**)$.   Now, if  $\psi:[X\setminus Y]^{\leq \omega}\rightarrow [Y]^{\leq \omega}$ is the function defined by $\psi(B)=\{\beta+1:\beta\in B\}$, for each $B\in [X\setminus Y]^{\leq \omega}$,  then $\psi$ is  $\omega$-monotone. We observe that for $B\in [X\setminus Y]^{\leq \omega}$, $cl_X(B)\setminus B\subseteq  cl_X(\psi(B))\subseteq r_{\psi(B)}([0,\kappa])$. As a consequence of Proposition \ref{novo2}, we have that $AD([0,\kappa])$ has an $r$-skeleton.
\end{example}

To finish this section we generalize Question 2.14 of the paper \cite{kalenda2} as follows.

\begin{question}
If $X$ is a linearly ordered compact space  with an $r$-skeleton, does $X$ admit an $r$-skeleton that satisfies the   conditions of the Theorem \ref{principal}?
\end{question}

\section{$\pi$-skeleton}

We remember that the Valdivia compact spaces are the  compact spaces in a cube  $[0,1]^\kappa$ with dense intersection with the $\Sigma$-product of $[0,1]^\kappa$.  As the family of compact spaces  with $r$-skeletons are the generalization of the Valdivia compact spaces, it is natural to ask about the possibility of  describing the compact spaces with an $r$-skeleton in  the vein of the definition of  Valdivia compact spaces. To get an answer to this question we introduce the following notion.

\begin{definition}
Let  $X\subseteq [0,1]^{T}$ be a compact space.
 A  family  $\{F_s:s\in \Gamma\}$ of cosmic spaces and an $
\omega$-monotone  function  $\psi:\Gamma\rightarrow [T]^{\leq\omega}$ generate a {\it $\pi$-skeleton} on $X$ if  the next conditions  are satisfied:
 \begin{itemize}
 \item[$(a)$]  $F_s\subseteq F_{t}$ whenever $s\leq t$,
 \item[$(b)$]  $\bigcup_{s\in \Gamma}F_s$ is dense in $X$, and
 \item[$(c)$]  $\pi_{\psi(s)}\upharpoonright_{F_s}$ is an homeomorphism such that $\pi_{\psi(s)}(X)=\pi_{\psi(s)}(F_s)$, for each $s\in \Gamma$.
 \end{itemize}
 The set $\bigcup_{s\in \Gamma}F_s$ will be called {\it the induced space} of the $\pi$-skeleton.
\end{definition}
In what follows,  when we will say  that $X$ has $\pi$-skeleton generated by the pair  $(\{F_s:s\in \Gamma\}, \psi:\Gamma\rightarrow [T]^{\leq\omega} )$ we shall understand that $X\subseteq [0,1]^T$.\\

Next, we shall describe two  examples of spaces with $\pi$-skeletons.

\begin{example}
If  $\alpha$ is an infinite cardinal number, then there exists an embedding $h:[0,\alpha]\rightarrow [0,1]^{\alpha + 1}$ such that $h([0,\alpha])$  admits a $\pi$-skeleton. Indeed, let  $h:[0,\alpha]\rightarrow [0,1]^{\alpha+1}$ be the function such that for every $\beta\leq \alpha$ we have
 $$
\pi_\theta(h(\beta))=\left\{ \begin{array}{lcc}
0 & if  &\theta< \beta\\
1 & if &\theta\geq\beta,
\end{array}
\right.
$$  for each $\theta<\alpha$. It is easy to verify that $h$ is an embedding. Let $L_\alpha$ be the set of limit points of $[0,\alpha]$ and   consider the set    $\Gamma=\{A\in [[0,\alpha]\setminus L_\alpha]^{\leq\omega}: 0\in A\}$. Let  $X=h([0,\alpha])$ and, for each $A\in \Gamma$ let $F_A:=\overline{h(A\cup A+1)}$ and   $\varphi(A):=A\cup\{\theta\leq\alpha:\theta+1\in A\mbox{ and } cof(\theta)\leq \omega\}$.  It is clearly that the  function $\varphi:\Gamma\rightarrow [\alpha+1]^{\leq\omega}$ is $\omega$-monotone and the family of cosmic subspaces  $\{F_A:A\in [Y]^{\leq \omega}\}$ satisfies $(a)$ and $(b)$. To prove the last condition of being $\pi$-skeleton, we fix $A\in \Gamma$.  First, we shall prove that $\pi_{\varphi(A)}(F_A)=\pi_{\varphi(A)}(X)$. Indeed, let $x\in X\setminus F_A$ and  choose $\beta \leq \alpha$ so that $x=h(\beta)$.  We point out that $\sup(\varphi(A))=\sup(A)$ and $\beta\notin h^{-1}(F_A)=\overline{A\cup A+1}$. Define $\beta_0:=\sup\{\theta\in \varphi(A):\theta<\beta\}$.\\

\begin{claim}{1} If $\beta_0\in \varphi(A)$, then $\beta_0+1<\beta$.
\end{claim}

\smallskip

\begin{proofclaim} It is clearly from  the choice of $\beta_0$ that $\beta_0<\beta$. Let us suppose that $\beta=\beta_0 + 1$.
Since $\beta_0\in \varphi(A)$, $\beta_0\in A$ or $\beta_0+1\in A$. Hence,  $\beta=\beta_0 +1\in A\cup A+1\subseteq h^{-1}(F_A) $. Since  $\beta\notin \overline{A\cup A+1}$, we must have that  $\beta_0+1<\beta$.
\end{proofclaim}

To continue with the proof we need to consider two cases for  $\beta_0$: For the first case, we suppose $\beta_0\in L_\alpha$. Clearly, $\beta_0\in \overline{\varphi(A)} \subseteq \overline{A\cup A+1}$. Since $\beta\notin \overline{A\cup A+1}$,   we have that    $\beta_0<\beta$ and $(\beta_0,\beta)\cap \varphi(A)=\emptyset$. Also, $\beta_0\notin \varphi(A)$, indeed, let us suppose  $\beta_0\in \varphi(A)$. Since $\beta_0\in L_\alpha$, we have that  $\beta_0+1\in A\subseteq \varphi(A)$. According to Claim 1, $\beta_0+1<\beta$, but this contradicts the equality $\beta_0=\sup\{\theta\in \varphi(A):\theta<\beta\}$.  Hence, if $\theta\in \varphi(A)$, then $\theta<\beta_0$ or $\theta\geq\beta$.  For $\theta\in \varphi(A)$ with $\theta< \beta_0$, we have $\pi_\theta(h(\beta_0))=0=\pi_\theta(h(\beta))$. And, for $\theta\in \varphi(A)$ with $\theta\geq \beta$, $\pi_\theta(h(\beta_0))=1=\pi_\theta(h(\beta))$. Thus, $\pi_{\varphi(A)}(h(\beta))=\pi_{\varphi(A)}(h(\beta_0))$ and $h(\beta_0)\in F_A$.
The second case is when $\beta_0\notin L_\alpha$. In this case, we have that $\beta_0\in \varphi(A)$. It then follows that $\beta_0<\beta$. By the Claim 1,  we have that $\beta_0\in A$.    For $\theta\in \varphi(A)$ with $\theta\leq \beta_0$, we have that $\pi_\theta(h(\beta_0+1))=0=\pi_\theta(h(\beta))$. And, for $\theta\in \varphi(A)$ with $\theta\geq \beta$, $\pi_\theta(h(\beta_0+1))=1=\pi_\theta(h(\beta))$. Thus, $\pi_{\varphi(A)}(h(\beta))=\pi_{\varphi(A)}(h(\beta_0+1))$ and $h(\beta_0+1)\in F_A$. In both cases, we obtain $\pi_{\varphi(A)}(h(\beta))\in \pi_{\varphi(A)}(F_A)$. Hence, $\pi_{\varphi(A)}(F_A)=\pi_{\varphi(A)}(X)$.\\

Let us prove that $\pi_{\varphi(A)}(F_A)$ is homeomorphic to $F_A$. It is enough prove that the function $\pi_{\varphi(A)}\upharpoonright_{F_A}$ is one-to-one. Let  $x,y\in F_A$ be distinct
and $\theta,\beta\in h^{-1}(F_A)$ such that $h(\theta)=x$ and $h(\beta)=y$. Without loss of generality we may assume that  $\theta<\beta$. Observe that  $\pi_{\beta'}(x)\neq \pi_{\beta'}(y)$ for each $\beta'\in [\theta,\beta)$.
  If $\beta$ is an isolated point, then there exist $\beta'<\alpha$ such that $\beta=\beta'+1$. So, we obtain that $\beta\in A\cup A+1$. If $\beta\in A$, then $\beta'\in \varphi(A)$. If $\beta\in A+1$, then $\beta'\in A\subseteq \varphi(A)$. If   $\beta$ is not isolated point, by density, then  $A\cap (\theta, \beta+1)\neq \emptyset$ and so there exist $\beta'\in A\cap [\theta,\beta)\subset \varphi(A)$. In all cases, $\varphi(A)\cap [\theta,\beta)\neq \emptyset$. Hence, we have $\pi_{\varphi(A)}(x)\neq \pi_{\varphi(A)}(y)$; that is,  $\pi_{\varphi(A)}$ one-to-one on $F_A$. This shows that  $\pi_{\varphi(A)}(F_A)$ is homeomorphic to $F_A$. Therefore,  $(\{F_A:A\in \Gamma\},\varphi)$ is a $\pi$-skeleton on $X$.
\end{example}

Our next example is the family of Valdivia compact spaces. To describe  a $\pi$-skeleton on a Valdivia compact space we need  prove, first, a useful technical lemma. We recall that for a power $[0,1]^T$ and a point $x\in X$, the support of $x$ is the set $supp(x):=\{t\in T: \pi_t(x)\neq 0\}$, if $A\subseteq [0,1]^T $, then $supp(A):=\bigcup_{x\in A} supp(x)$.

\begin{lemma}\label{lemmahomeo1}
Let $Y\subseteq [0,1]^T$. If  $A\in [Y]^{\leq\omega}$, then $\pi_{supp(A)}\upharpoonright_{\overline{A}}$ is a one-to-one function.
\end{lemma}

\begin{proof} First, we shall prove that $supp(x)\subseteq supp(A)$ whenever   $A\in [Y]^{\leq\omega}$ and $x\in \overline{A}$.   If $x\in A$, we trivially have that  $supp(x)\subseteq supp(A)$. Let $x\in \overline{A}\setminus A$ such that  $ supp(x)\setminus supp(A)\not= \emptyset$. For a fix  $t\in supp(x)\setminus supp(A)$ there is an open subset $U\subseteq [0,1]$ such that $0\notin U$ and $x\in \pi^{-1}_{t}(U)$.  So, we have that $A\cap \pi^{-1}_{t}(U)=\emptyset$, which contradicts that  $x\in \overline{A}$. Hence, for every $x\in \overline{A}$ we have $supp(x)\subseteq supp(A)$.  Let $x,y\in \overline{A}$ with $x\neq y$. Then, there exists $t\in supp(x)\cup  supp(y)$ such that $\pi_t(x)\neq \pi_t(y)$. Since $supp(x)\cup supp(y)\subseteq supp(A)$, we have that $\pi_{supp(A)}(x)\neq \pi_{supp(A)}(y)$. Thus, we have proved that $\pi_{supp(A)}\upharpoonright_{\overline{A}}$ is a one-to-one function.
\end{proof}

\begin{proposition}\label{Valdiviapies}  If $X$ is a Valdivia compact space, then $X$ admit a $\pi$-skeleton.
 \end{proposition}

 \begin{proof} Suppose $X$ is a compact space of $[0,1]^T$ and  $Y=X\cap \Sigma[0,1]^T$  is dense in $X$. We shall prove that $X$ admits a $\pi$-skeleton.

 \smallskip

\begin{claim}{1} The set $\Gamma=\{A\in [Y]^{\leq \omega}: \pi_{supp(A)}(X)=\pi_{supp(A)}(\overline{A})\}$ is cofinal in $[Y]^{\leq\omega}$.
\end{claim}

\smallskip

\begin{proofclaim} Fix $A\in [Y]^{\leq\omega}$. We know that $\pi_{supp(A)}(X)$ is a separable space. Hence, it is possible to find $D'\in [Y]^{\leq\omega}$ such that $\pi_{supp(A)}(D')$ is dense in $\pi_{supp(A)}(X)$. We put $D_1=D'\cup A$ and for a positive $n<\omega$  suppose that we have   established the set $D_n$. Choose $D'_{n+1}\in [Y]^{\leq\omega}$ so that $\pi_{supp(D_n)}(D'_{n+1})$ is dense in $\pi_{supp(D_n)}(X)$. We put $D_{n+1}=D'_{n+1}\cup D_n$. Thus, we have defined  an increasing set $\{D_n:n<\omega\}\subseteq [Y]^{\leq\omega}$. We consider $D=\bigcup_{n<\omega}D_n$ and  shall prove that $\pi_{supp(D)}(D)$ is dense on $\pi_{supp(D)}(X)$. Let $x\in X$, $t\in supp(D)$ and $U$ be an open subset of $[0,1]$ with $\pi_t(x)\in U$. Since $t\in supp(D)$ and $supp(D)=\bigcup_{n<\omega}supp(D_n)$, there is $n<\omega$ such that $t\in supp(D_n)$. As $\pi_{supp(D_n)}(D'_{n+1})$ is dense in $\pi_{supp(D_n)}(X)$, there is $d\in D'_{n+1}\subseteq D_{n+1}\subseteq D$ such that $\pi_t(d)\in U$. Hence, $\pi_{supp(D)}(D)$ is dense in $\pi_{supp(D)}(X)$. By Lemma \ref{lemmahomeo1}, $\pi_{supp(D)}\upharpoonright_{\overline{D}}$ is an homeomorphism and so we have that $\pi_{supp(D)}(\overline{D})=\pi_{supp(D)}(X)$. By construction, $A\subseteq D$ and thus we conclude that $\Gamma$ is cofinal in $[Y]^{\leq\omega}$.
\end{proofclaim}

\begin{claim}{2}
The set $\Gamma$ under the order given by the contention  is a $\sigma$-complete directed partially ordered set.
\end{claim}

\smallskip

\begin{proofclaim}
It is immediate to see that $\Gamma$ is directed partially ordered set. To  prove that $\Gamma$ is $\sigma$-complete we consider an increasing sequence $\{A_n\}_{n<\omega}$  in $\Gamma$, $A=\bigcup_{n<\omega}A_n$ and $x\in X$.  For each $n<\omega$, using that $A_n\in \Gamma$, we pick $a_n\in \overline{A_n}$ such that $\pi_{supp(A_n)}(x)=\pi_{supp(A_n)}(a_n)$. Without loss of generality,  we suppose  $a=\lim_{n\rightarrow \infty}a_n$ since $Y$ is Frechet-Uryshon. As   $\{a_n\}_{n<\omega}\subseteq \overline{\bigcup_{n<\omega}A_n}=\overline{A}$,  $a\in \overline{A}$.  By the choice of $\{a_n\}_{n<\omega}$ and the continuity of $\pi_{supp(A)}$, we have
$$
\pi_{supp(A)}(x)=\lim_{n\rightarrow\infty}\pi_{supp(A)}(a_n)= \pi_{supp(A)}(\lim_{n\rightarrow\infty}a_n)=\pi_{supp(A)}(a).
$$
\end{proofclaim}

\smallskip

 Finally, for each  $A\in \Gamma$ we define $F_A:=\overline{A}$. The family $\{F_A:A\in \Gamma\}$ consists of cosmic spaces which satisfy $(a)$ and $(b)$. Now, we define $\varphi:\Gamma\rightarrow [T]^{\leq \omega}$ by $\varphi(A)=supp(A)$, for each $A\in \Gamma$. Thus, we conclude that the pair $(\{F_A:A\in \Gamma\},\varphi)$ is a $\pi$-skeleton on $X$.
\end{proof}

\begin{remark}
We observe that if    $\{F_s:s\in \Gamma\}$ is a  family of cosmic spaces  on a compact space $X$ which satisfies $(a)$, $(b)$  and $F_s\subseteq \Sigma_0$, for every $s\in \Gamma$, then $X$ is a Valdivia compact space. Hence and according to Proposition \ref{Valdiviapies}, we have that  a compact space $X$ is Valdivia iff admits a $\pi$-skeleton $(\{F_s:s\in \Gamma\},\varphi)$ such that  $F_s\subseteq \Sigma_0$, for every $s\in \Gamma$.
\end{remark}

To prove the main result of this section, we require some  technical lemmas.  In the first of these lemmas, we will see that the notion   of $\pi$-skeleton can be generalized when the base space is a subspace of a product of arbitrary cosmic compact spaces.

\begin{lemma}\label{obs12}
Let $X$ be a compact subspace of a product of cosmic spaces $\Pi_{t\in T}Z_t$. Su\-ppose that there exists a pair $(\{F_s:s\in \Gamma\},\psi:\Gamma\rightarrow [T]^{\leq\omega})$ where $\{F_s:s\in \Gamma\}$ is a family which satisfies $(a)$ and $(b)$, $\psi$ is $\omega$-monotone and for each $s\in \Gamma$ we have that $\pi_{\psi(s)}\upharpoonright_{F_s}$ is an homeomorphism such that $\pi_{\psi(s)}(X)=\pi_{\psi(s)}(F_s)$. Then there is an embedding $g:X\rightarrow [0,1]^{T'}$ such that $g(X)$ admits a $\pi$-skeleton.
\end{lemma}
\begin{proof} For each $t\in T$, fix  a countable set $T_t$ and an embedding $g_t:Z_t\rightarrow [0,1]^{T_t}$. We consider $g:\Pi_{t\in T}Z_t\rightarrow \Pi_{t\in T}[0,1]^{T_t}$ given by $g((x_t)_{t\in T})=(g_t(x_t))_{t\in T}$, for each $x\in \Pi_{t\in T}Z_t$. It is easy to verify that   $g$ is an embedding.  Without loss of generality, we suppose that $T_t\cap T_{t'}=\emptyset$, for distinct  $t,t'\in T$.  Let $T'=\bigcup_{t\in T}T_t$ and  $\phi:[T]^{\leq\omega}\rightarrow [T']^{\leq\omega}$ be the mapping given by $\phi(A)=\bigcup_{t\in A}T_t$, for each $A\in [T]^{\leq \omega}$. It is immediate from its definition to see that $\phi$  is $\omega$-monotone.  We shall prove that the pair $(\{g(F_s):s\in \Gamma\},\phi\circ\psi)$ is a $\pi$-skeleton on $g(X)$. Indeed, it is clearly that  $\{g(F_s):s\in \Gamma\}$ is a family of cosmic spaces which satisfies $(a)$ and $(b)$, and  since composition of $\omega$-monotone functions is $\omega$-monotone, the composition $\phi\circ \psi$ is an $\omega$-monotone function. It remains prove  condition  $(c)$ for the pair $(\{g(F_s):s\in \Gamma\},\phi\circ\psi)$.  To prove that  $\pi_{\phi\circ \psi(s)}(g(X))=\pi_{\phi\circ\psi(s)}(g(F_s))$, for each $s\in \Gamma$, we fix $s
\in \Gamma$ and $x\in X$. By hypothesis there is $y\in F_s$ such that $\pi_{\psi(s)}(x)=\pi_{\psi(s)}(y)$.
 So, $\pi_t(x)=\pi_t(y)$ and $g_t(\pi_t(x))=g_t(\pi_t(y))$, for each $t\in \psi(s)$. Hence, $\pi_{\phi\circ\psi(s)}(g(x))=\pi_{\phi\circ\psi}(g_t(\pi_t(x))_{t\in T})=(g_t(\pi_t(x)))_{t\in\psi(s)}= (g_t(\pi_t(y)))_{t\in\psi(s)}=\pi_{\phi\circ\psi(s)}(g(y))$. So, we conclude that $\pi_{\phi\circ\psi(s)}(g(X))=\pi_{\phi\circ\psi(s)}(g(F_s))$. Now, we shall prove that $\pi_{\phi\circ\psi(s)}\upharpoonright_{g(F_s)}$ is one-to-one. We pick  $x, y\in F_s$ with $x\neq y$. Using that $\pi_{\psi(s)}\upharpoonright_{F_s}$ is an homeomorphism, we choose $t\in \psi(s)$ such that $\pi_t(x)\neq\pi_t(y)$. Since  $g_t$ is an embedding,  we have that $g_t(\pi_t(x))\neq g_t(\pi_t(y))$ and then we  conclude that $\pi_{\phi\circ\psi(s)}(g(x))=(g_t(\pi_t(x)))_{t\in\psi(s)}\neq (g_t(\pi_t(y)))_{t\in\psi(s)}=\pi_{\phi\circ\psi(s)}(g(y))$. As a consequence of injectivity of $\pi_{\phi(\psi(s))}\upharpoonright_{g(F_s)}$, we have that $\pi_{\phi(\psi(s))}\upharpoonright_{g(F_s)}$ is an homeomorphism. Therefore, $(\{g(F_s):s\in \Gamma\},\phi\circ\psi)$ is a
  $\pi$-skeleton on $g(X)$.
\end{proof}

 We observe that the reciprocal of  Lemma \ref{obs12} is true. Hence, to prove that a space $X$ admits a $\pi$-skeleton it is enough to embed $X$ in an arbitrary product of cosmic spaces and to prove the existence of a pair as in the hypothesis of Lemma \ref{obs12}.\\

The next two lemmas are  consequences of the  characterization of spaces with $r$-skeletons  using inverse limits. The first Lemma is an implicit result inside of the work of inverse limits and Valdivia compact spaces in the paper \cite{kubis1}.

\begin{lemma}\label{cen1}
Let $X$ be a compact space. If $X$ admits an $r$-skeleton $\{r_s:s\in \Gamma\}$,  then  $\underleftarrow{lim}\langle r_s(X),r^{t}_s ,\Gamma\rangle =\{(r_s(x))_{s\in \Gamma}:x\in X\}$, where $r^t_s=r_s\upharpoonright_{r_t(X)}$, for each $t\geq s$.
\end{lemma}

\begin{lemma}[\cite{kubis1}]\label{salceno}
Let $X$ be a compact space. If $X$ admits an $r$-skeleton $\{r_s:s\in \Gamma\}$,  then  $X=\underleftarrow{lim}\langle r_s(X),r^{t}_s ,\Gamma\rangle$, where $r^t_s=r_s\upharpoonright_{r_t(X)}$, for each $t\geq s$. In particular, if $A$ is a countable directed subset of $\Gamma$, then $$r_{\sup(A)}(X)=\underleftarrow{lim}\langle r_{s}(X),r^t_s,A\rangle.$$

\end{lemma}

Next, we state the main result.

\begin{theorem}\label{teoprin}
Let $X$ be a compact space. Then, $X$ admit an $r$-skeleton iff admit a $\pi$-skeleton.
\end{theorem}
\begin{proof}
First, we shall prove the necessity. Let $\{r_s:s\in \Gamma\}$ be an $r$-skeleton on $X$ and ,by  Corollary \ref{cen1} and Lemma \ref{salceno}, we may assume that   $X=\underleftarrow{lim}\langle r_s(X),r^{t}_s ,\Gamma\rangle= \{(r_s(x))_{s\in \Gamma}:x\in X\}$,  where $r^t_s=r_s\upharpoonright_{r_t(X)}$, for each $t\geq s$.  We shall prove that $X$ admits a $\pi$-skeleton.  Let $\Gamma'$ be  the family of  countable  directed subsets of $\Gamma$. Observe that $\Gamma'$ is a  $\sigma$-complete up-directed partially ordered set with the order given by set containment. For each $A\in \Gamma'$, let $F_A:=r_{sup(A)}(X)$ and  define $\varphi:\Gamma'\rightarrow [\Gamma]^{\leq\omega}$   by $\varphi(A)=A$, for each $A\in \Gamma'$. It is easy to see that the function $\varphi$ is $\omega$-monotone and the family  $\{F_A:A\in \Gamma'\}$ of cosmic subspaces   of $X$ satisfies $(a)$.  We note that $\bigcup_{A\in \Gamma'}r_{sup(A)}(X)$  is the induced set by the $r$-skeleton on $X$ because of $\{\{s\}:s\in \Gamma\}\subseteq\Gamma'$, and we deduce  that $\{F_A:A\in \Gamma'\}$  satisfies  property $(b)$. Let $A\in \Gamma'$. By using property $(ii)$ of the $r$-skeleton definition, we obtain that
 $$\pi_A(F_A)=\{(r_s(r_{sup(A)}(x)))_{s\in A}:x\in X\}=\{(r_s(x))_{s\in A}:x\in X\}=\pi_A(X).$$
 Hence, $\pi_{A}(X)=\pi_A(F_A)$.  Now, by Corollary \ref{cen1} and Lemma \ref{salceno}, we know that $F_A=r_{sup(A)}(X)=\underleftarrow{lim}\langle r_s(X),r^t_s,A\rangle=\{(r_s(x))_{s\in A}:x\in F_A\}$, it then  follows that $\pi_A\upharpoonright_{F_A}$ is one-to-one. Thus, we conclude that  $\pi_A\upharpoonright_{F_A}$ induces an homeomorphism between $F_A$ and $\pi_A(X)$. Hence, the pair  $(\{F_A:A\in \Gamma'\},\varphi)$ satisfies  the hypothesis of Lemma \ref{obs12} and so $X$ admits  a $\pi$-skeleton.

 \medskip

Now, we shall prove the sufficiency.  Assume that   $X\subseteq [0,1]^{T}$ and    $(\{F_s:s\in \Gamma\},\varphi)$ is  a $\pi$-skeleton on $X$ with induced space  $Y$. For each $A\in [Y]^{\leq\omega}$, let $r_A:X\rightarrow F_A$  be the function defined by $r_A:=\pi^{-1}_{\varphi(A)}\upharpoonright_{F_A}\circ\pi_{\varphi(A)}$. As $\pi_{\varphi(A)}(F_A)$ is homeomorphic to $F_A$ and $\pi_{\varphi(A)}(X)\subseteq \pi_{\varphi(A)}(F_A)$, we have that $r_A$ is a retraction.
 We assert  that $\{r_A:A\in [Y]^{\leq\omega}\}$ is an $r$-skeleton on $X$. Indeed,   condition $(i)$ is satisfied because $r_A(X)=F_A$ is homeomorphic to $\pi_{\varphi(A)}(F_A)$ and $\pi_{\varphi(A)}(F_A)$ is a cosmic space. For the conditions $(ii)$ and $(iii)$ we consider the next claim.\\

\begin{claim}{1} For every $A\in [Y]^{\leq \omega}$, $\pi_{\varphi(A)}(r_A(x))=\pi_{\varphi(A)}(x)$ for all $x\in X$.
\end{claim}

\medskip

\begin{proofclaim} Fix $A\in [Y]^{\leq \omega}$. By the definition of $r_A$,  we have that  $$\pi_{\varphi(A)}(r_A(x))=\pi_{\varphi(A)}(\pi^{-1}_{\varphi(A)}\upharpoonright_{F_A}(\pi_{\varphi(A)}(x)))=(\pi_{\varphi(A)}\circ\pi^{-1}_{\varphi(A)}\upharpoonright_{F_A})\circ\pi_{\varphi(A)}(x)=\pi_{\varphi(A)}(x),$$
for each $x\in X$.
\end{proofclaim}

 Let us prove condition $(ii)$. Fix  $A, B\in [Y]^{\leq\omega}$ such that $A\subseteq B$ and  fix  $x\in X$.   Since   $r_A(x)\in F_A\subseteq F_B$ and $r_B$ is the identity on $F_B$, we have that $r_B(r_A(x))=r_A(x)$. Now, we shall prove that $r_A(x)=r_A(r_B(x))$.  By  Claim 1, $\pi_{\varphi(B)}(x)=\pi_{\varphi(B)}(r_B(x))$ and  $\pi_{\varphi(A)}(r_A(r_B(x)))=\pi_{\varphi(A)}(r_B(x))$, for each $x\in X$. As $\varphi(A)\subseteq \varphi
(B)$, we have that $\pi_{\varphi(A)}(r_A(r_B(x)))=\pi_{\varphi(A)}(r_B(x))=\pi_{\varphi(A)}(x)$. On the other hand, according  to Claim 1, we have that  $\pi_{\varphi(A)}(r_A(x))=\pi_{\varphi(A)}(x)$. Hence, $\pi_{\varphi(A)}(r_A(r_B(x)))=\pi_{\varphi(A)}(r_A(x))$. Since $\pi_{\varphi(A)}$ is injective on $F_A$, we  conclude that $r_A(x)= r_A(r_B(x))$. Thus, $r_A=r_A\circ r_B= r_B\circ r_A$  and $(ii)$ hold. To prove condition $(iii)$ we consider the next claim.\\

In what follows, $\mathcal{B}$ will be a countable base of $[0,1]$.

\medskip

\begin{claim}{2}\label{4obs} If  $A\in [Y]^{\leq\omega}$, then
 $$\{(\pi_{\varphi(A)}\upharpoonright_{F_A})^{-1}(\bigcap_{t\in G}[t,V_t]\cap \pi_{\varphi(A)}(F_A)): G\in [\varphi(A)]^{<\omega}\mbox{ and } V_t\in \mathcal{B}\}$$ is base for $F_A$.
\end{claim}
\medskip
\begin{proofclaim}
It is evident that  $\{\bigcap_{t\in G}[t,V_t]: G\in [\varphi(A)]^{<\omega}\mbox{ and } V_t\in \mathcal{B}\}$ is a base of the space $[0,1]^{\varphi(A)}$ and, so $\{\bigcap_{t\in G}[t,V_t]\cap \pi_{\varphi(A)}(F_A): G\in [\varphi(A)]^{<\omega}\mbox{ and } V_t\in \mathcal{B}\}$ is a base for the space  $\pi_{\varphi(A)}(F_A)$. Since $\pi_{\varphi(A)}\upharpoonright_{F_A}$ is an homeomorphism, we obtain that  $\{(\pi_{\varphi(A)}\upharpoonright_{F_A})^{-1}(\bigcap_{t\in G}[t,V_t]\cap \pi_{\varphi(A)}(F_A)): G\in [\varphi(A)]^{<\omega}\mbox{ and } V_t\in \mathcal{B}\}$ is a base for $F_A$.
 \end{proofclaim}

For the condition $(iii)$, we consider an increasing sequence $\langle A_n\rangle_{n<\omega}\subseteq [Y]^{\leq\omega}$. Put $A=\bigcup_{n<\omega}A_n$ and fix $x\in X$. By Claim 2, we may  choose $G\in [\varphi(A)]^{<\omega}$ and $\{V_t:t\in G\}\subseteq \mathcal{B} $ such that $r_A(x)\in (\pi_{\varphi(A)}\upharpoonright_{F_A})^{-1}(\bigcap_{t\in G}[t,V_t]\cap\pi_{\varphi(A)}(F_A))$.  Since $\varphi(A)=\bigcup_{n<\omega}\varphi(A_n)$, there is $n_0<\omega$ such that $G\subseteq \varphi(A_n)$, for each $n\geq n_0$. According to Claim 1, we know that $\pi_{\varphi(A)}(r_A(x))=\pi_{\varphi(A)}(x)$ and $\pi_{\varphi(An)}(r_{A_n}(x))=\pi_{\varphi(An)}(x)$, for each $n\geq n_0$. In particular, as $G\subseteq \varphi(A_n)\subseteq \varphi(A)$, $\pi_{G}(r_A(x))=\pi_{G}(x)=\pi_{G}(r_{A_n}(x))$, for each $n\geq n_0$. Since $r_A(x)\in (\pi_{\varphi(A)}\upharpoonright_{F_A})^{-1}(\bigcap_{t\in G}[t,V_t]\cap\pi_{\varphi(A)}(F_A))$, it then follows that  $\pi_t(r_A(x))\in V_t$, for each $t\in G$. Using that $\pi_{G}(r_A(x))=\pi_{G}(x)=\pi_{G}(r_{A_n}(x))$, we have that $\pi_t(r_{A_n}(x))\in V_t$, for each $t\in G$. Hence, for every $n\geq n_0$, $r_{A_n}(x)\in \bigcap_{t\in G}[t,V_t]$ and thus we have that condition $(iii)$ is satisfied.

\medskip

For to establish the condition $(iv)$, we note that $r_A(X)=F_A$ and $\overline{\bigcup_{A\in [Y]^{\leq\omega}}r_A(X)}=\overline{\bigcup_{A\in [Y]^{\leq\omega}}F_A}=X$.
Therefore, we have that $\{r_A:A\in [Y]^{\leq\omega}\}$ is an $r$-skeleton on $X$.
\end{proof}

Following the definition of  the $r$-skeletons, we say that a $\pi$-skeleton on a compact space $X$ is {\it full} if its induced space is  the space $X$. Hence, we have the next consequence of the previous theorem.
\begin{corollary}
Let $X$ be a compact space. $X$ is a Corson space iff $X$ admits a full $\pi$-skeleton.
\end{corollary}

\begin{proof}
In \cite{cuth1}, the authors proved that $X$ is a Corson space iff $X$ admits a full $r$-skeleton. From the proof of Theorem \ref{teoprin} we have that the  spaces induced by the $r$-skeleton and by a $\pi$-skeleton coincided. Hence, we deduce that $X$ admits a full $r$-skeleton iff admits a full $\pi$-skeleton.
\end{proof}

Next, we will give a proof of the stability of  $\pi$-skeletons under the product of spaces with  a $\pi$-skeleton.

\begin{theorem}\label{preserprod}
The product of spaces with $\pi$-skeleton admit a $\pi$-skeleton.
\end{theorem}

\begin{proof}
Let $\{X_i:i\in I\}$ be a family of compact spaces which admit  $\pi$-skeletons. For each $i\in I$, let $T_i$ be such that $X_i\subseteq [0,1]^{T_i}$,   $\mathcal{F}_i=\{F^{i}_s: s\in \Gamma_i\}$ be a family of cosmic closed subspaces of $X_i$ with induced space $Y_i$ and $\varphi_i:\Gamma_i\rightarrow [T_i]^{\leq\omega}$ be the $\omega$-monotone function such that $(\mathcal{F}_i,\varphi_i)$   is a $\pi$-skeleton on $X_i$.  Without loss of  generality, we suppose that the sets $T_i$'s are pairwise disjoints. For each $i\in I$, let $\leq_{i}$ be the order on $\Gamma_i$ and for every $A\subseteq I$ and $S_1, S_2\in \Pi_{i\in J}\Gamma_i$,   $S_1\leq_{A} S_2$ will mean that $\pi_{i}(S_1)\leq_{i} \pi_{i}(S_2)$, for every $i\in A$.

\medskip

To prove that $\Pi_{i\in I}X_i$ admit a $\pi$-skeleton we fix $y\in \Pi_{i\in I}Y_i$ which  will help us in the process of the proof. Define  $\Gamma:=\{(A,S):A\in [I]^{\leq \omega} \mbox{ and } S\in \Pi_{i\in A}\Gamma_i\}$. For $(A_1,S_1),(A_2,S_2)\in \Gamma$, we define $(A_1,S_1)\leq_{p}(A_2,S_2)$ if  $A_1\subseteq A_2 $, $S_1\leq_{A_1}\pi_{A_1}(S_2)$ and $\pi_i(y)\in F^{i}_{\pi_{i}(S_2)}$, for each $i\in A_2\setminus A_1$.

\medskip

First, we shall prove that $(\Gamma,\leq_{p})$ is a $\sigma$-complete up-directed partially ordered set. The reflexivity and antisymmetry  of $\leq_p$ are immediate from the orders of the sets $\Gamma_i$. For the transitivity of $\leq_p$, let $(A_1,S_1),(A_2,S_2),(A_3,S_3)\in \Gamma$ such that  $(A_1,S_1)\leq_p(A_2,S_2)$ and $(A_2,S_2)\leq_p (A_3,S_3)$. By definition of $\leq_p$, we have that $S_1\leq_{A_1}\pi_{A_1}(S_2)$ and $S_2\leq_{A_2}\pi_{A_2}(S_3)$. Using that $A_1\subseteq A_3$, $\pi_{A_1}(S_2)\leq_{A_1}\pi_{A_1}(S_3)$ and the transitivity of the orders of the sets $\Gamma_i$, for each $i\in A_1$, we obtain that $S_1\leq_{A_1}\pi_{A_1}(S_3)$. Now, let $i\in A_3\setminus A_1$, if $i\in A_3\setminus A_2$, then using that $(A_2,S_2)\leq_p (A_3,S_3)$ we have $\pi_{i}(y)\in F^{i}_{\pi_i(S_3)}$. Now, if $i\in A_2\setminus A_1$, then  $\pi_{i}(y)\in F^{i}_{\pi_i(S_2)}$. Since $\pi_{i}(S_2)\leq_{i} \pi_{i}(S_3)$ and that $\mathcal{F}_i$ satisfies $(a)$, we have that  $F^{i}_{\pi_{i}(S_2)}\subseteq  F^{i}_{\pi_{i}(S_3)}$ and $\pi_{i}(y)\in F^{i}_{\pi_{i}(S_3)}$. Thus, for each $i\in A_3\setminus A_1$, $\pi_i(y)\in F^{i}_{\pi_i(S_3)}$. Hence, $(A_1,S_1)\leq_p (A_3,S_3)$.
Now, we will prove that $\Gamma$ is an up-directed set. Let $(A_1,S_1),(A_2,S_2)\in \Gamma$.  For each $i\in A_1\cap A_2$, using that $\Gamma_i$ is up-directed,  fix $s_i\in \Gamma_i$ such that $s_i\geq_{i}\pi_i(S_1)$ and $s_i\geq_{i} \pi_i(S_2)$. For $i\in A_2\setminus A_1$, let  $s_i$ be such that $s_i\geq_{i} \pi_{i}(S_2)$ and $\pi_i(y)\in F^{i}_{s_i}$. For $i\in A_1\setminus A_2$, let  $s_i$ be such that $s_i\geq_{i} \pi_{i}(S_1)$ and $\pi_i(y)\in F^{i}_{s_i}$. Let $S\in \Pi_{i\in A_1\cup A_2}\Gamma_i$ such that $\pi_i(S)=s_i$, for each $i\in A_1\cup A_2$. We have that $(A_1\cup A_2,S)\in \Gamma$ and it is clear that $(A_1,S_1)\leq_p(A_1\cup A_2,S)$ and $(A_2,S_2)\leq_p(A_1\cup A_2,S)$.
For the $\sigma$-completeness, let $\langle (A_n,S_n)\rangle_{n<\omega}\subseteq \Gamma$ be an increasing sequence and let $A=\bigcup_{n<\omega}A_n$. For each $i\in A$, let $n_i=\min\{n<\omega:i\in A_n\}$,  $s_i=\sup_{n\geq n_i}\pi_i(S_n)$; and  define $S\in \Pi_{i\in A}\Gamma_i$ by $\pi_i(S)=s_i$, for each $i\in A$. We have that $(A,S)\in \Gamma$ and claim  that  $(A,S)=\sup_{n<\omega}(A_n,S_n)$. Indeed, first we prove that $(A_n,S_n)\leq_p(A,S)$, for each $n<\omega$. Fix $n<\omega$. By definition,  we have that $A_n\subseteq A$. By the choice of $S$, we have that $S_n\leq_{A_n}\pi_{A_n}(S)$. For $i\in A\setminus A_n$, there exists $n'\geq n$ such that $i\in A_{n'}$. Using that $(A_n,S_n)\leq_p(A_{n'},S_{n'})$, we have that $\pi_i(y)\in F^i_{\pi_i(S_{n'})}$.
 Since $\mathcal{F}_i$ satisfies $(a)$ and $\pi_i(S_{n'})\leq_i \pi_i(S)$, we deduce that $F^i_{\pi_i(S_{n'})}\subseteq F^i_{\pi_i(S)}$. Hence, $\pi_{i}(y)\in F^i_{\pi_i(S)}$ and $(A_n,S_n)\leq_p (A,S)$. Now, let $(A',S')\in \Gamma$ be such that $(A_n,S_n)\leq_p (A',S')$, for each $n<\omega$. We shall prove that $(A,S)\leq_p (A',S')$. Since $A_n\subseteq A'$, for each $n<\omega$, $A\subseteq A'$. For each $i\in A$, we have that $\pi_i(S)=\sup_{n\geq n_i}\pi_i(S_n)$ and $\pi_i(S_n)\leq_i \pi_i(S')$, for all $n\geq n_i$. It then follows that $\pi_i(S)\leq_i \pi_i(S')$ and $S\leq_A \pi_A(S')$. For $i\in A'\setminus A$, is easy to see $\pi_(y)\in F^i_{\pi_i(S')}$. We conclude that $(A,S)\leq_p(A',S')$ and $(A,S)=\sup_{n<\omega}(A_n,S_n)$.

 \medskip

 For each $(A,S)\in \Gamma$, let $F_{(A,S)}:=\Pi_{i\in A}F^{i}_{\pi_{i}(S)}\times \Pi_{i\in I\setminus A}\{\pi_i(y)\}$. The family $\mathcal{F}=\{F_{(A,S)}:(A,S)\in \Gamma\}$ of cosmic closed subspaces of $\Pi_{i\in I}X_i$ satisfies $(a)$. From the density of $Y_i$ in $X_i$, for each $i\in I$, we  deduce  that $\bigcup\mathcal{F}$ is dense in $\Pi_{i\in I}X_i$. Let $\varphi:\Gamma\rightarrow [\bigsqcup_{i\in I}T_i]^{\leq\omega}$ be defined by $\varphi((A,S))=\bigcup_{i\in A}\varphi_{i}(\pi_{i}(S))$, for each $(A,S)\in \Gamma$. We have that $\varphi$ is  $\omega$-monotone. Indeed, let $(A_1,S_1),(A_2,S_2)\in \Gamma$, such that $(A_1,S_1)\leq_p (A_2,S_2)$. We have that $\pi_{A_1}(S_1)\leq_p \pi_{A_1}(S_2)$. Hence, $\varphi((A_1,S_1))=\bigcup_{i\in A_1}\varphi_i(\pi_i(S_1))\subseteq \bigcup_{i\in A_2}\varphi_i(\pi_i(S_2))=\varphi((A_2,S_2))$. Let $\langle(A_n,S_n)\rangle_{n<\omega}$ be an increasing sequence in $\Gamma$ and $(A,S)=\sup_{n<\omega}(A_n,S_n)$. It is easy to see that $\varphi((A_n,S_n))\subseteq \varphi((A,S))$, for each $n<\omega$,  it then follows that $ \bigcup_{n<\omega}\varphi((A_n,S_n))\subseteq \varphi((A,S))$. Let $i\in A$ and $n_i=\min\{n<\omega:i\in A_n\}$,  we know that $\pi_i(S)=\sup\{\pi_i(S_n): n\geq n_i\}$. Using that $\varphi_i$ is $\omega$-monotone, we have that $\varphi_i(\pi_i(S))=\bigcup_{n\geq n_i}\varphi_i(\pi_i(S_n))\subseteq \bigcup_{n\geq n_i}\varphi((A_n,S_n))$.  Since    $\varphi((A,S))=\bigcup_{i\in A}\varphi_{i}(\pi_{i}(S))$, we deduce that  $\varphi((A,S))\subseteq\bigcup_{n<\omega}\varphi((A_n,S_n))$. Hence, $\varphi((A,S))=\bigcup_{n<\omega}\varphi((A_n,S_n))$; that is, $\varphi$ is $\omega$-monotone. Finally, for every $(A,S)\in \Gamma$, we shall prove that $\pi_{\varphi((A,S))}\upharpoonright_{F_{(A,S)}}$ is an homeomorphism.  Note  that
 $$\pi_{\varphi((A,S))}(X)=\Pi_{i\in A}\pi_{\varphi_{i}(\pi_{i}(S))}(X_i)=\Pi_{i\in A}\pi_{\varphi_{i}(\pi_{i}(S))}(F^{i}_{\pi_{i}(S)})=\pi_{\varphi((A,S))}(F_{(A,S)}).
 $$ Since $\pi_{\varphi_{i}(\pi_{i}(S))}\upharpoonright_{F^{i}_{\pi_{i}(S)}}$ is an homeomorphism, it  is clear that
$\pi_{\varphi((A,S))}\upharpoonright_{F_{(A,S)}}$ is an homeo\-morphism. Therefore, $(\mathcal{F}=\{F_{(A,S)}:(A,S)\in \Gamma\},\varphi)$ is a $\pi$-skeleton on  $ \Pi_{i\in I}X_i$.
\end{proof}

The previous theorem provides  an alternative  proof of the fact that the product of spaces with $r$-skeleton admits a $r$-skeleton, this proof does not use the theory  of elementary submodels as  it was done in \cite{cuth1}.

\begin{corollary}\label{preguntapi}
The  arbitrary product of compact spaces with $r$-skeleton admit an $r$-skeleton.
\end{corollary}
\begin{proof}
This is a consequence of Theorem \ref{teoprin} and Theorem \ref{preserprod}.
\end{proof}

Concerning the reciprocal of Theorem \ref{preserprod} we may ask the following.

 \begin{question}
 If $X_1$ and $X_2$ are  compact spaces such that $X_1\times X_2$ admits a $\pi$-skeleton,  must either $X_1$ or $X_2$ admit a $\pi$-skeleton?
 \end{question}

The question when $X_1$ and $X_2$ are  compact Valdivia spaces is posed in \cite[Q. 2.3]{kalenda2}.

\medskip

Finally, we point out that the notion of $\pi$-skeleton somehow resembles the definition of Valdivia spaces using the notion of $\Sigma$-product.

\section{weak $c$-skeletons}


Our first task of this section is to weaken the notion of  a  $c$-skeleton:

\medskip

For a family $\{F_s:s\in \Gamma\}$ of closed spaces on a compact space $(X,\tau)$ we recall the following conditions  used in the previous section:
 \begin{itemize}
 \item[$(a)$] $F_s\subseteq F_{t}$, whenever $s\leq t$ and
 \item[$(b)$] $\bigcup_{s\in \Gamma}F_s$ is dense in $X$.
 \end{itemize}
If  $\mathcal{B}$ is a base for a topology on $X$ and $x \in X$,  then define $\mathcal{N}(\mathcal{B},x):=\{U\in \mathcal{B}:x\in U\}$. A function  $\psi:\Gamma \rightarrow [\tau]^{\leq\omega}$ is called an {\it assignment of bases for}  $\{F_s:s\in \Gamma\}$  if  $\psi(s)$ is a base for a topology on $X$, $\psi(s)\upharpoonright_{F_s}:=\{U\cap F_s:U\in \psi(s)\}$ is a base for $F_s$ and $\emptyset\notin \psi(s)\upharpoonright_{F_s}$.

\begin{definition}\label{wcesq}
Let $(X,\tau)$ be a space. A pair $(\{F_s:s\in \Gamma\},\psi)$ is a \textit{weak $c$-skeleton} on $X$ if the following conditions hold:
\begin{enumerate}
\item\label{cond0}  $\{F_s:s\in \Gamma\}$ is a family of cosmic  subspaces of $X$ which satisfies $(a)$ and $(b)$;

\item\label{cond01} $\psi$ is an assignment of bases for  $\{F_s:s\in \Gamma\}$; and

\item \label{condition1} if $x\in X$ and  $z\in \bigcap_{U\in \mathcal{N}(\psi(s),x)}\overline{U\cap F_s}$, then for each $V\in \mathcal{N}(\psi(s),z)$ there is $U\in \mathcal{N}(\psi(s),x)$ such that $\overline{U\cap F_s}\subseteq V$.
\end{enumerate}
 The {\it induced space} of a weak $c$-skeleton is $\bigcup_{s\in \Gamma}F_s$.
\end{definition}

Let us  see how from a $c$-skeleton  we can obtain a weak $c$-skeleton:  Consider a $c$-skeleton $\{(F_s,\mathcal{B}_s):s\in \Gamma\}$ on a space $X$. If we define $\psi(s) = \mathcal{B}_s$ for every $s \in \Gamma$, then it is evident that $\psi$ is an assignment of bases for  $\{F_s:s\in \Gamma\}$. Conditions (\ref{cond0}) and (\ref{cond01}) of a weak $c$-skeleton easily hold.  The existence of a  Tychonoff space $Z_s$ and continuous map $g_s:(X,\psi(s))\rightarrow Z_s$ which separates the points of $F_s$ implies  the third Condition, for each $s\in\Gamma$.

\medskip
The next technical lemma shows that  a closed subset of a compact space  is a retract whenever we have a base which satisfies the condition (\ref{condition1}).
\begin{lemma}\label{lemaconti}
Let $(X,\tau)$ be a compact space and  $F\subseteq X$ be a cosmic closed subspace. Then the next conditions are equivalents:

\begin{itemize}
\item There is  a base  $\mathcal{B}\in [\tau]^{\leq\omega}$ for a topology on $X$ such that  $\mathcal{B}\upharpoonright_{F}$ is a base for $F$, $U\cap F\neq \emptyset$ for each $U\in \mathcal{B}$ and $F$ and $\mathcal{B}$ satisfies the condition (\ref{condition1})  of the definition of weak $c$-skeleton.

\item There exist a retraction  $r:X\rightarrow F$.
\end{itemize}

\end{lemma}
\begin{proof} Suppose that $\mathcal{B}\in [\tau]^{\leq\omega}$ is a base for a topology on $X$ such that  $\mathcal{B}\upharpoonright_{F}$ is a base for $F$, $U\cap F\neq \emptyset$ for each $U\in \mathcal{B}$ and $F$ and $\mathcal{B}$ satisfies the condition (\ref{condition1})  of definition of weak $c$-skeleton.   Let $x\in X$. First, we will prove that $\bigcap_{V\in \mathcal{N}(\mathcal{B},x)}\overline{V}\cap F\neq\emptyset$. Since $\mathcal{B}$ is a base for a topology on $X$, we have that $\mathcal{N}(\mathcal{B},x)\neq \emptyset$. As $U\cap F\neq \emptyset$, for every $U\in \mathcal{N}(\mathcal{B},x)$, we have that  $\{\overline{U}\cap F: U\in \mathcal{N}(\mathcal{B},x)\}$ is a nonempty family of nonempty closed subsets. We claim that $\{\overline{U}\cap F: U\in \mathcal{N}(\mathcal{B},x)\}$ has the finite intersection property. In fact, let $U_1,U_2\in \mathcal{N}(\mathcal{B},x)$. Since $\mathcal{B}$ is a base for a topology on $X$, there is $U_3\in \mathcal{N}(\mathcal{B},x)$ such that $U_3\subseteq U_1\cap U_2$. Hence, $\emptyset\neq \overline{U_3}\cap F\subseteq \overline{U_1\cap U_2}\cap F$. As  $\{\overline{U}\cap F: U\in \mathcal{N}(\mathcal{B},x)\}$  has the finite intersection property, by the compactness  of $X$, we have that  $\bigcap_{U\in \mathcal{N}(\mathcal{B},x)}\overline{U}\cap F\neq \emptyset$. Now, we will show that $|\bigcap_{U\in \mathcal{N}(\mathcal{B},x)}\overline{U}\cap F|=1$. Pick two distinct points $y_0,y_1\in \bigcap_{U\in \mathcal{N}(\mathcal{B},x)}\overline{U}\cap F$. Since  $F$ is Hausdorff and $\mathcal{B}\upharpoonright_{F}$ is base for $F$, there are $W_0\in \mathcal{N}(\mathcal{B},y_0)$ and $W_1\in \mathcal{N}(\mathcal{B},y_1)$ such that $W_0\cap W_1\cap F=\emptyset$. Using condition  (\ref{condition1}) of the Definition \ref{wcesq}, there are $V_0,V_1\in \mathcal{N}(\mathcal{B},x)$ such that $\overline{V_0\cap F}\subseteq W_0$ and $\overline{V_1}\cap F\subseteq W_1$. However, we have that  $\emptyset\neq \bigcap_{U\in \mathcal{N}(\mathcal{B},x)}\overline{U}\cap F \subseteq \overline{V_0}\cap F\cap \overline{V_1}\cap F \subseteq W_0\cap W_1\cap F$, which contradicts that $W_0$ and $W_1$ are disjoints sets in $F$. Hence, $|\bigcap_{U\in \mathcal{N}(\mathcal{F},x)}\overline{U\cap F}|=1$. Now, consider  the function  $r:X\rightarrow F$ such that  $r(x)$ is the unique point of $ \bigcap_{V\in \mathcal{N}(\mathcal{B},x)}\overline{V}\cap F$, for every $x\in X$. From the definition of $r$,  it is easy to note that if $x\in F$ and  $V\in\mathcal{B}$, then $r(x)=x$ and $r(V)\subseteq \overline{V}\cap F$.  Let us show  that $r$ is continuous.  Fix $U\in \mathcal{B}$. Remember that $\mathcal{B}\upharpoonright_{F}$ is  base for the subspace $F$. We assert that $r^{-1}(U\cap F)$ is an open subset of $X$. Indeed, fix $x\in X$ such that $r(x)\in U\cap F$. According to condition (\ref{condition1}) of Definition \ref{wcesq}, we may choose  $V\in \mathcal{N}(\mathcal{B},x)$ with $\overline{V}\cap F\subseteq U$. Since $r(V)\subseteq \overline{V}\cap F$, we have that $V\subseteq r^{-1}(\overline{V}\cap F)\subseteq r^{-1}(U\cap F)$. Thus, $r$ is a continuous function. Finally,   $r$ is a retraction because $r\upharpoonright_{F}$ is the identity.

\medskip

Now, assume that $r:X\rightarrow F$ is a retraction. Let $\mathcal{B}'$ be a countable base for the space $F$ and $\mathcal{B}:=\{r^{-1}(V):V\in \mathcal{B}'\}$. We note that $\mathcal{B}\in [\tau]^{\leq\omega}$ is a base for a topology on $X$, $\mathcal{B}\upharpoonright_{F}$ is a base for $F$ and $U\cap F\neq \emptyset$, for each $U\in \mathcal{B}$. It remains to show that the third condition of the definition of weak $c$-skeleton is satisfied. Indeed, fix   $ x\in X$.  First, we will prove that $r(x)\in\bigcap_{V\in \mathcal{N}(\mathcal{B},x)}\overline{V\cap F}$. To have this done choose $V\in \mathcal{N}(\mathcal{B},x)$ and $U\in \mathcal{N}(\mathcal{B},r(x))$. We assert that $U\cap V\cap F\neq \emptyset$. Certainly, we can find $V',U'\in \mathcal{B}'$ such that $V=r^{-1}(V')$ and $U=r^{-1}(U')$. Then, we have that  $r(x)\in V'$, and since $r$ is a retraction,  $r(x)\in U'$. Hence, $r(x)\in U'\cap V'$. As $\mathcal{B}'$ is a base in $F$, there exist $W'$ such that $r(x)\in W'\subseteq U'\cap V'$. Thus, $x\in r^{-1}(W')\subseteq r^{-1}(U'\cap V')= U\cap V$. As $r^{-1}(W')\in\mathcal{B}$, we obtain that $\emptyset\neq r^{-1}(W')\cap F\subseteq U\cap V\cap F$. It then follows that $r(x)\in \overline{V}\cap F$ and so  $r(x)\in\bigcap_{V\in \mathcal{N}(\mathcal{B},x)}\overline{V\cap F}$.
 Next, we will show that $|\bigcap_{U\in \mathcal{N}(\mathcal{B},x)}\overline{U\cap F}|=1$. Suppose that  $y\in \bigcap_{U\in \mathcal{N}(\mathcal{B},x)}\overline{U\cap F}$ and $y\neq r(x)$. Since  $F$ is Hausdorff and $\mathcal{B}\upharpoonright_{F}$ is base for $F$, there are $W_0\in \mathcal{N}(\mathcal{B},y)$ and $W_1\in \mathcal{N}(\mathcal{B},r(x))$ such that $W_0\cap W_1\cap F=\emptyset$. As $r$ is a retraction, $r(x)\in r(W_1)=W_1'$. Hence, $x\in r^{-1}(W_1')=W_1\in\mathcal{B}$. Since $y\in \overline{W_1}\cap F$, we have that $W_0\cap W_1\cap F\neq \emptyset$, but this contradicts the choice of $W_0$ and $W_1$.  Hence, $|\bigcap_{U\in \mathcal{N}(\mathcal{B},x)}\overline{U\cap F}|=1$ and $\{r(x)\}=\bigcap_{V\in \mathcal{N}(\mathcal{B},x)}\overline{V\cap F}$.  Finally,  to prove condition (\ref{condition1}). Consider   $r^{-1}(W)\in \mathcal{N}(\mathcal{B},r(x))$ where $W\in \mathcal{B}$. As $r$ is a retraction, $r(x)=r(r(x))\in W$. Choose $W'\in \mathcal{B}$ with $r(x)\in W'\subseteq \overline{W'}\subseteq W$. Hence, $x\in r^{-1}(W')\subseteq r^{-1}(\overline{W'})\subseteq r^{-1}(W)$. Thus, $\overline{r^{-1}(W')\cap F}\subseteq r^{-1}(W)$ and condition (\ref{condition1}) is satisfied for the point $x$.
\end{proof}

\begin{remark}\label{observacionlemaconti}
Implicitly in the proof of   Lemma \ref{lemaconti} there is a useful equality that we will use in the proofs of our next results: $\{r(x)\}=\bigcap_{V\in \mathcal{N}(\mathcal{B},x)}\overline{V}\cap F$, where $r: X \to F$ is the retraction obtained in the previous lemma.
\end{remark}

\medskip

In the next result, we show that the notion of $r$-skeleton implies the notion of weak $c$-skeleton.

  \begin{theorem}\label{weakteo1}
  If $X$ is a compact space with an $r$-skeleton $\{r_s:s\in \Gamma \}$, then $X$ admits a weak $c$-skeleton $(\{F_s:s\in \Gamma \},\psi)$ such that  $\{r_s(x)\}=\bigcap_{U\in \mathcal{N}(\psi(s),x)}\overline{U}\cap F$, for each $x\in X$ and each $s\in \Gamma$.
 \end{theorem}

 \begin{proof}
   Let $\{r_s:s\in \Gamma\}$ be an $r$-skeleton on $X$ where $Y$ is its induced space. For each $s\in \Gamma$, put $F_s=r_s(X)$, choose   a countable base $\mathcal{B}_s$ of $F_s$ and  define $\psi(s):=r^{-1}_s(\mathcal{B}_s)$. Consider this function $\psi:\Gamma\rightarrow [\tau]^{\leq\omega}$. From Lemma \ref{lemaconti} we have that   $\psi:\Gamma\rightarrow [\tau]^{\leq\omega}$ satisfies conditions (\ref{cond01}) and  (\ref{condition1}). As $\{r_s:s\in \Gamma\}$ is an $r$-skeleton, the conditions $(i)$, $(ii)$ and $(iv)$ of the $r$-skeleton definition  imply that $\{F_s:s\in \Gamma\}$ is a family of cosmic spaces which satisfies $(a)$ and $(b)$. Therefore, $(\{F_s:s\in\Gamma\},\psi)$ is a weak $c$-skeleton.
 \end{proof}

The next example is a compact space which admits a weak $c$-skeleton but does not admit an $r$-skeleton. This space is one of the examples of non-Valdivia compact spaces, given  by O. Kalenda in \cite{kalenda1}, to show that the  continuous images of Valdivia compact spaces are not necessarily Valdivia compact spaces.

\begin{example}
Let $X$ the space obtained from $[0,\omega_1]\times\{0,1\}$  by identifying the points $(\omega_1,0)$ and  $(\omega_1,1)$ and denote by $\tau_X$ the quotient topology on $X$. First, we shall prove that $X$ admits a weak $c$-skeleton. Let $\mathcal{B}$ a base for the space $[0,\omega_1]$. The set $\Gamma :=\omega_1\times \omega_1$ is a $\sigma$-complete up-directed partially ordered set with the order defined by  $(\alpha,\beta)\leq_{\Gamma}(\alpha',\beta')$ if $\alpha\leq \alpha'$ and $\beta \leq \beta'$. For every $\alpha<\omega_1$, fix a subset $\mathcal{B}_\alpha\in [\mathcal{B}]^{\leq \omega}$ such that $\{U\cap [0,\alpha]:U\in \mathcal{B}_\alpha\}$ is a base for the subspace $[0,\alpha]$ and if $\alpha\notin U\in \mathcal{B}_\alpha$, then $U\subseteq [0,\alpha)$. For $\alpha<\omega_1$ and  $i\in \{0,1\}$, let $\mathcal{B}_\alpha\times \{i\}:=\{U\times \{i\}:U\in \mathcal{B}_\alpha\}$.
For each, $(\alpha,\beta)\in \Gamma$, define $F_{(\alpha,\beta)}:=([0,\alpha]\times\{0\})\cup ([0,\beta])\times\{1\}$ and
$$
\psi((\alpha,\beta)):= (\mathcal{B}_\alpha\times\{0\}) \cup (\mathcal{B}_\beta\times\{1\}) \cup\{\big( (U\cup (\alpha,\omega_1])\times\{0\}\big)\cup \big ( (W\cup (\beta,\omega_1])\times\{1\}\big):
 $$
 $$
U\in \mathcal{B}_\alpha, \alpha\in U, W\in\mathcal{B}_\beta \mbox{ and } \beta\in W\}.
$$
We assert that $(\{F_{(\alpha,\beta)}:(\alpha,\beta)\in \Gamma\},\psi)$ is a weak $c$-skeleton on $X$. It is easy to see that $\{F_{(\alpha,\beta)}:(\alpha,\beta)\in\Gamma\}$ satisfies the condition (\ref{cond0}). Now, we shall prove that $\psi$ satisfies condition (\ref{cond01}). Fix $(\alpha,\beta)\in \Gamma$.
By the choice of the sets $\mathcal{B}_\alpha$ and $\mathcal{B}_\beta$, we have that $\psi((\alpha,\beta))\upharpoonright_{F_{(\alpha,\beta)}}$ is a base for the space $F_{(\alpha,\beta)}$ and $\emptyset\notin \psi((\alpha,\beta))\upharpoonright_{F_{(\alpha,\beta)}}$.  We shall prove that $\psi((\alpha,\beta))$ is a base for a topology on $X$. Indeed, by assumption, $\psi((\alpha,\beta))$ covers  $X$. Let $U, U'\in \mathcal{B}_\alpha$ and $W, W'\in \mathcal{B}_\beta$. Let us consider three cases:
\begin{itemize}

\item[Case 1] If $x\in (U\times\{0\})\cap (U'\times\{0\})$ and $\alpha\notin U\cap U'$,  since $\mathcal{B}_{\alpha}$ is base for $[0,\alpha]$, then we have that  there exist $V\in \mathcal{B}_\alpha$ such that $x\in (V\times\{0\})\subseteq  (U\times\{0\})\cap (U'\times\{0\})$ and $V\times\{0\}\in \psi((\alpha,\beta))$.
\item[Case 2] When $x\in (W\times\{1\})\cap(W'\times\{1\})$ and $\beta\notin W\cap W'$ we proceed in a  similar way as in the previous case.
\item [Case 3]Let us suppose $x\in (((U\cup (\alpha,\omega_1])\times\{0\})\cup((W\cup (\beta,\omega_1])\times\{1\}))\cap (((U'\cup (\alpha,\omega_1])\times\{0\})\cup((W'\cup (\beta,\omega_1])\times\{1\}))$. As $\mathcal{B}_\alpha$ and $\mathcal{B}_\beta$ are bases for the spaces $[0,\alpha]$ and $[0,\beta]$, respectively, we can find  $V\in \mathcal{B}_\alpha$ and $V'\in \mathcal{B}_\alpha$ such that either $x\in (V\cup(\alpha,\omega_1])\times\{0\}\subseteq ((U\cup(\alpha,\omega_1])\times\{0\})\cap ((U'\cup(\alpha,\omega_1])\times\{0\})$ or  $x\in (V'\cup (\beta,\omega_1])\times\{1\}\subseteq ((W\cup(\beta,\omega_1])\times\{1\})\cap ((W'\cup(\beta,\omega_1])\times\{1\})$. Thus, $x\in ((V\cup(\alpha,\omega_1])\times\{0\})\cup ((V'\cup (\beta,\omega_1])\times\{1\}) \subseteq (((U\cup (\alpha,\omega_1])\times\{0\})\cup((W\cup (\beta,\omega_1])\times\{1\}))\cap (((U'\cup (\alpha,\omega_1])\times\{0\})\cup((W'\cup (\beta,\omega_1])\times\{1\}))$.
\end{itemize}
This shows that  $\psi((\alpha,\beta))$ is a base for a topology on $X$ and so the condition (\ref{cond01}) holds.
    It remains to show the condition (\ref{condition1}). Fix $(\alpha,\beta)\in \Gamma$ and $x\in X$. First, we suppose  $x\in F_{(\alpha,\beta)}$. It is easy to see that $\bigcap_{U\in \mathcal{N}(\psi((\alpha,\beta)),x)} \overline{U\cap F_{(\alpha,\beta)}}=\{x\}$, and condition (\ref{condition1}) follows trivially. Now, let us suppose that $x\notin F_{(\alpha,\beta)}$. We have that   $\bigcap_{U\in \mathcal{N}(\psi((\alpha,\beta)),x)} \overline{U\cap F_{(\alpha,\beta)}}=\{(\alpha,0),(\beta,1)\}$. Note that if $z\in \{(\alpha,0),(\beta,1)\}$ and  $U\in \mathcal{N}(\psi((\alpha,\beta)),z)$, then $U\in \mathcal{N}(\psi((\alpha,\beta)),x)$ and so   condition \ref{condition1} holds. Therefore, $(\{F_{(\alpha,\beta)}:(\alpha,\beta)\in \Gamma\},\psi)$ is a weak $c$-skeleton on $X$. We know that $X$ cannot be a Valdivia compact space and does not admit an $r$-skeleton, because every compact space with weight $\aleph_1$ and $r$-skeleton is a Valdivia compact space (see \cite{kubis1}).
\end{example}

In the next theorem, we establish conditions on a weak $c$-skeleton in order to be equivalent to an $r$-skeleton. Thus we obtain a characterization of spaces that admit an $r$-skeleton in terms of a special weak $c$-skeleton.

\begin{theorem}\label{teoprincipalceqsresq} For a compact space $X$, the following statements are equivalent:
\begin{itemize}
\item  There is a weak $c$-skeleton  $(\{F_s:s\in\Gamma \}, \psi)$ on $X$ which satisfies
\begin{itemize}
\item[$(*)$] For every $s\leq t$, $x\in X$, $z\in \bigcap_{U\in \mathcal{N}(\psi(t),x)}\overline{U\cap F_t}$ and  $V\in \mathcal{N}(\psi(s),x)$, there is $W\in \mathcal{N}(\psi(s),z)$ such that $\overline{W\cap F_s}\subseteq \overline{V\cap F_s}$; and

\item[$(**)$] if $\langle s_n \rangle_{n<\omega}$ is an increasing sequence in $\Gamma$ with $s
=\sup\langle s_n\rangle_{n<\omega}$, then for each $x\in X$ and $U\in \mathcal{N}(\psi(s),x)$ there is $m<\omega$ such that $\overline{U\cap F_s}\cap {V\cap F_{s_n}}\neq \emptyset$ for every  $n\geq m$ and for every $V\in \mathcal{N}(\psi(s_n),x)$.
\end{itemize}
\item  There is an $r$-skeleton on $X$.
\end{itemize}
\end{theorem}

\begin{proof} Let  $(\{F_s:s\in\Gamma \}, \psi)$ be a weak $c$-skeleton that satisfies $(*)$ and $(**)$.
For each  $s\in \Gamma$, consider  the retraction $r_s$ given in  Lemma \ref{lemaconti}. We shall prove that $\{r_s: s\in \Gamma\}$ is an $r$-skeleton on $X$. In fact, the condition $(i)$  is true because $r_s(X)=F_s$ and $F_s$ is a cosmic space, for each $s\in \Gamma$.
For condition $(ii)$, fix $s, t \in \Gamma$ with  $s\leq t$. On one hand, we know that $r_s(X)=F_s\subseteq F_t=r_t(X)$ and $r_t$ is the identity on $F_t$, hence,  $r_t\circ r_s=r_s$. On the other hand, fix $x\in X$.
Using  condition $(*)$, for every $V\in \mathcal{N}(\psi(s),x)$  choose $W_V\in \mathcal{N}(\psi(s),r_{t}(x)) $ such that $\overline{W_V\cap F_s}\subseteq \overline{V\cap F_s}$. It then follows that  $\bigcap_{V\in \mathcal{N}(\psi(s),x)}\overline{W_V\cap F_s}\subseteq \bigcap_{V\in \mathcal{N}(\psi(s),x)}\overline{V\cap F_s}$. So $\bigcap_{W\in \mathcal{N}(\psi(s),r_t(x))}\overline{W\cap F_s}\subseteq\bigcap_{V\in \mathcal{N}(\psi(s),x)}\overline{W_V\cap F_s}$ and this implies that
$$
\{r_s(r_t(x))\}=\bigcap_{W\in \mathcal{N}(\psi(s),r_t(x))}\overline{W\cap F_s}\subseteq \bigcap_{V\in \mathcal{N}(\psi(s),x)}\overline{V\cap F_s}=\{r_s(x)\}.
 $$
 Thus, $r_s\circ r_t=r_s$ and $(ii)$ holds.
For condition $(iii)$, let $\langle s_n\rangle_{n<\omega}$ be an increasing sequence in $\Gamma$ and $s=\sup\langle s_n\rangle_{n<\omega}$. Fix $x\in X$ and $U\in \mathcal{N}(\psi(s),r_s(x))$. Choose $V\in \mathcal{N}(\psi(s),r_s(x))$ such that $\overline{V}\cap F_s\subseteq U\cap F_s$. For condition $(**)$,  pick $m<\omega$ such that for every $n\geq m$  we have that $\overline{V\cap F_s}\cap {W\cap F_{s_n}}\neq \emptyset$, for every $W\in \mathcal{N}(\psi(s_n),r_s(x))$. Fix $n\geq m$. By property $(**)$, it is easy to see that $\{\overline{V\cap F_s}\cap {W\cap F_{s_n}}: W\in \mathcal{N}(\psi(s_n),r_s(x)) \}$ has the finite intersection property, so $\{\overline{V\cap F_s}\cap \overline{W\cap F_{s_n}}: W\in \mathcal{N}(\psi(s_n),r_s(x)) \}$ has the finite intersection property. Since $F_s$ is a compact space,  we have that $\overline{V\cap F_s}\cap \big(\bigcap_{W\in \mathcal{N}(\psi(s_n),r_s(x))}\overline{W\cap F_{s_n}}\big)\neq \emptyset$. As $\{r_{s_n}(x)\}=\{r_{s_n}(r_s(x))\}= \bigcap_{W\in \mathcal{N}(\psi(s_n),r_s(x))}\overline{W\cap F_{s_n}}$, we obtain that $r_{s_n}(x)\in \overline{V\cap F_s}\subseteq U$. Thus, for each $n\geq m$, $r_{s_n}(x)\in U$; that  is, $r_{s_n}(x)\rightarrow r_s(x)$ in $F_s$.  Hence, $r_{s_n}(x)\rightarrow r_s(x)$ in $X$ and $(iii)$ holds.
  The last  condition $(iv)$ follows from the fact that  $\{F_s:s\in \Gamma\}$  satisfies condition $(b)$.  Thus we have shown that $\{r_s: s\in \Gamma\}$ is an $r$-skeleton on $X$.

\medskip

Assume that $\{r_s:s\in \Gamma\}$ is an $r$-skeleton on $X$ and $Y$ is the induced space. For each $s\in \Gamma$,   define $F_s:=r_s(X)$, choose   a countable base $\mathcal{B}_s$ of $F_s$ and put  $\psi(s):=r^{-1}_s(\mathcal{B}_s)$. Theorem \ref{weakteo1} asserts that $(\{F_s:s\in\Gamma \},\psi)$ is a weak $c$-skeleton on $X$ with induced set $Y$.  To establish property $(*)$, let $s\leq t$, $x\in X$ and $W\in \mathcal{B}_s$ such that $x\in r^{-1}_s(W)$. According to Theorem \ref{weakteo1}, we know that $\{r_t(x)\}=\bigcap_{U\in \mathcal{N}(\psi(t),x)}\overline{U\cap F_t}$.  Since $\{r_s:s\in \Gamma\}$ is an $r$-skeleton, we have that $r_s(r_t(x))=r_s(x)\in W$. Hence, $r_t(x)\in r^{-1}_s(W)$ and $(*)$ follows trivially. For property $(**)$, let $\langle s_n\rangle_{n<\omega}\subseteq \Gamma$ be an increasing sequence and suppose that $s=\sup\langle s_n\rangle_{n<\omega}$, $x\in X$ and $W\in \mathcal{B}_s$ satisfy that $x\in r^{-1}_s(W)$. As $\{r_s:s\in\Gamma\}$ is an $r$-skeleton, $r_{s_n}(x)\rightarrow r_s(x)$. Since $r_s(x)\in W$, we have that there is $m<\omega$, such that for every $n\geq m$, $r_{s_n}(x)\in W$. Let $n\geq m$ and $r^{-1}_{s_n}(W')\in \mathcal{N}(\psi(s_n),x)$, where  $W'\in \mathcal{B}_{s_n}$. Since $\{r_{s_n}(x)\}=\bigcap_{U\in \mathcal{N}(\psi(s_{n}),x)}\overline{U\cap F_{s_n}}$, we have that  $r_{s_n}(x)\in \overline{r^{-1}_{s_n}(W')\cap F_{s_n}}$.  As $r_{s}(r_{s_n}(x))=r_{s_n}(x)$ and $r_{s_n}(x)\in W$ , we have that  $r_{s_n}(x)\in r^{-1}_{s}(W)$. Using that $r_{s_n}(x)\in \overline{r^{-1}_{s_n}(W')\cap F_{s_n}}$, we have that $r^{-1}_{s}(W)\cap r^{-1}_{s_n}(W')\cap F_{s_n}'\neq \emptyset$.   Since $r^{-1}_{s}(W)\cap r^{-1}_{s_n}(W')\cap F_{s_n}=r^{-1}_{s}(W)\cap F_s\cap r^{-1}_{s_n}(W')\cap F_{s_n}$,   we  conclude that $\overline{r^{-1}_s(W)\cap F_s}\cap r^{-1}_{s_n}(W')\cap F_s \neq \emptyset$ and then $(**)$ holds. Therefore, $(\{F_s:s\in \Gamma\},\psi)$ is a weak $c$-skeleton on $X$ which satisfies $(*)$ and $(**)$.
\end{proof}

The Corollary 5.6 of the paper \cite{casa1} shows that the notion of full $c$-skeleton is equivalent to the notion of full $r$-skeleton; that is, the Corson compact spaces are the spaces which admits a full $c$-skeleton. This result motivates the next  characterization of a  Valdivia compact space by mean of a special  weak $c$-skeleton.

\begin{theorem} For a compact space  $X$, the following statements are equivalent:
\begin{itemize}
\item  There is a weak $c$-skeleton $(\{F_s:s\in\Gamma \}, \psi)$ on $X$  which satisfies
\begin{itemize}

\item[$(**)$] if $\langle s_n \rangle_{n<\omega}$ is an increasing sequence in $\Gamma$ with $s
=\sup\langle s_n\rangle_{n<\omega}$, then for each $x\in X$ and $U\in \mathcal{N}(\psi(s),x)$ there is $m<\omega$ such that $\overline{U\cap F_s}\cap {V\cap F_{s_n}}\neq \emptyset$ for every  $n\geq m$ and for every $V\in \mathcal{N}(\psi(s_n),x)$; and

\item[$(***)$] for every $s,t\in \Gamma$, $x\in X$ and  $V\in \mathcal{N}(\psi(s),r_t(x))$, there is $U_V\in \mathcal{N}(\psi(t),r_s(x))$ such that $\emptyset\neq \overline{U_V}\cap F_t\cap F_s\subseteq \overline{V}\cap F_s$; and $\{\overline{U_V}\cap F_t\cap F_s: V\in  \mathcal{N}(\psi(s),r_t(x))\}$ has the finite intersection property.
\end{itemize}
\item  There is a commutative $r$-skeleton on $X$.
\end{itemize}
\end{theorem}

\begin{proof}
Let us suppose  $(\{F_s:s\in\Gamma \}, \psi)$ is  a weak $c$-skeleton on $X$ which satisfies $(**)$ and $(***)$. According the proof of the Theorem \ref{teoprincipalceqsresq}, $(\{F_s:s\in\Gamma \}, \psi)$ induces a family of retractions $\{r_s:s\in \Gamma\}$ on $X$ which satisfies $(i)$ and $(iv)$ of the $r$-skeleton definition, $\{r_s(x)\}=\bigcap_{V\in \mathcal{N}(\psi(s),x)}\overline{V\cap F_s}$ for each $s\in \Gamma$ and each $x\in X$; and  $r_s=r_t\circ r_s$ provided that $s\leq t$. Next, to prove  the commutativity  of the retractions (also, this implies $(ii)$) fix $s,t\in \Gamma$. From $(***)$ we have that  for each $V\in \mathcal{N}(\psi(s),r_t(x))$ we can find  $U_V\in \mathcal{N}(\psi(t),r_s(x))$ so that $\overline{U_V}\cap F_t\cap F_s\subseteq \overline{V}\cap F_s$. As  $\{\overline{U_V}\cap F_t\cap F_s:V\in \mathcal{N}(\psi(s),r_t(x))\}$ has the finite intersection property and $\bigcap_{V\in \mathcal{N}(\psi(t),r_s(x)) }\overline{U_V}\cap F_t\cap F_s\subseteq \bigcap_{V\in \mathcal{N}(\psi(t),r_s(x)) }\overline{V}\cap F_s$, we have that $\bigcap_{V\in \mathcal{N}(\psi(t),r_s(x)) }\overline{U_V}\cap F_t\cap F_s=\{r_t(r_x(x))\}$. Now, the equality $\bigcap_{U\in \mathcal{N}(\psi(s),r_t(x)) }\overline{U}\cap  F_s = \bigcap_{V\in \mathcal{N}(\psi(t),r_s(x)) }\overline{U_V}\cap F_t\cap F_s $ implies  that $\{r_s(r_t(x))\}=\bigcap_{V\in \mathcal{N}(\psi(t),r_s(x)) }\overline{U_V}\cap F_t\cap F_s $. Hence, $r_s(r_t(x))=r_t(r_s(x))$. This proves the commutativity of the retractions and we also obtain condition $(ii)$. Using the condition $(ii)$ and $(**)$ and some ideas from the  proof in the Theorem \ref{teoprincipalceqsresq}, we obtain $(iii)$.
Therefore,  $\{r_s:s\in \Gamma\}$ is a commutative $r$-skeleton.
\medskip

Now, let us suppose that $\{r_s:s\in \Gamma\}$ is a commutative $r$-skeleton. For each  $s\in \Gamma$  define  $F_s:=r_s(X)$, choose   a countable base $\mathcal{B}_s$ of $F_s$ and  define $\psi(s):=r^{-1}_s(\mathcal{B}_s)$. From Theorem \ref{teoprincipalceqsresq} we have  that $(\{F_s:s\in\Gamma \},\psi)$ is a weak $c$-skeleton on $X$ with induced set $Y$ which satisfy conditions   $(*)$ and $(**)$. It remains to prove condition $(***)$. Let $s,t\in \Gamma$ and $x\in X$.  Fix $r^{-1}_s(W),r^{-1}_s(V)\in \mathcal{N}(\psi(s),r_t(x))$ so that $r_t(x)\in r^{-1}_s(W)\subseteq \overline{r^{-1}_s(W))}\subseteq r^{-1}_s(V)$. We have that $r_s(r_t(x))\in W$. Since $r_s$ is a retraction  and $W\subseteq F_s$, we obtain that  $r_s(r_s(r_t(x)))\in W$. Hence, $r_s(r_t(x))\in r^{-1}_s(W)$. By the commutativity of the $r$-skeleton, we obtain that $ \bigcap_{U\in \mathcal{N}(\psi(t),r_s(x)) }\overline{U}\cap F_t=\{r_t(r_s(x))\}=\{r_s(r_t(x))\}=\bigcap_{W'\in \mathcal{N}(\psi(s),r_t(x)) }\overline{W'}\cap F_s\subseteq r^{-1}_s(W)\cap F_s$. That is, $\bigcap_{U\in \mathcal{N}(\psi(t),r_s(x)) }\overline{U}\cap F_t\subseteq r^{-1}_s(W)\cap F_s$.  Since $r_t(r_s(x))=r_s(r_t(x))$, we have that $\bigcap_{U\in \mathcal{N}(\psi(t),r_s(x)) }\overline{U}\cap F_t\cap F_s=\bigcap_{U\in \mathcal{N}(\psi(t),r_s(x)) }\overline{U}\cap F_t\subseteq r^{-1}_s(W)\cap F_s$. By the compactness  of $X$, there is $k<\omega$ and $\{U_0,...,U_k\}\subseteq \mathcal{N}(\psi(t),r_s(x))$ such that $\bigcap^{k}_{i=0}\overline{U_i}\cap F_t\cap F_s\subseteq r^{-1}_s(W)\cap F_s$. Since $\psi(t)$ is a base on  $X$, there is $U_V\in \mathcal{N}(\psi(t),r_s(x))$ such that $\overline{U_V}\cap F_t\cap F_s\subseteq \bigcap^{k}_{i=0}\overline{U_i}\cap F_t\cap F_s$. Hence, $\overline{U_V}\cap F_t\cap F_s\subseteq r^{-1}_s(W)\cap F_s$. As $\bigcap_{V\in \mathcal{N}(\psi(t),r_s(x)) }\overline{U_V}\cap F_t\cap F_s =\{r_s(r_t(x))\}$, we have that $\{\overline{U_V}\cap F_t\cap F_s: V\in  \mathcal{N}(\psi(s),r_t(x))\}$ has the finite intersection property.
\end{proof}

In the paper \cite{casa1}, the authors proved that for a compact space $X$, if $X$ has a (full) $c$-skeleton, then $C_p(X)$ has a (full) $q$-skeleton. And, if $X$ has a (full) $q$-skeleton, then $C_p(X)$ has a (full) $c$-skeleton.  In this sense, we have the next problem.

\begin{question}
Let  $X$ be a compact space. Is there  a  notion weaker than the notion  of a $q$-skeleton on $C_p(X)$ that is related to the notion of a weak $c$-skeleton on $X$?
\end{question}

In the following two diagrams, we illustrate the connection among all  the topological properties that we considered in this article.

\medskip
The first diagram is about the general case:
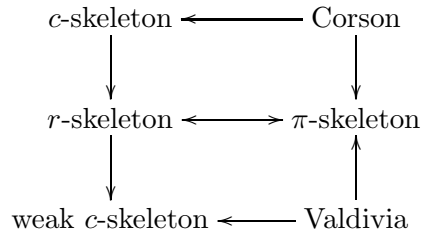
\begin{displaymath}
    \xymatrix{ c\mbox{-skeleton} \ar[d] & \mbox{Corson} \ar[l]\ar[d] \\
               r\mbox{-skeleton} \ar[d]\ar[r] & \pi\mbox{-skeleton}\ar[l]\\
               \mbox{weak } c\mbox{-skeleton}& \mbox{Valdivia}\ar[l] \ar[u]}
\end{displaymath}

This second diagram is a particular case when the skeletons are full.
\begin{displaymath}
    \xymatrix{\mbox{full } c\mbox{-skeleton} \ar[d]\ar[r] & \mbox{Corson} \ar[l]\ar[d] \\
              \mbox{full } r\mbox{-skeleton}\ar[u] \ar[r] & \mbox{full } \pi\mbox{-skeleton}\ar[l]\ar[u]  }
\end{displaymath}


\end{document}